\documentclass{article}

\usepackage{amsmath,amssymb,amsthm,amscd,mathtools}
\usepackage{graphicx}
\usepackage{longtable}
\usepackage[square]{natbib}

\usepackage[letterpaper,vmargin=3cm]{geometry}

\usepackage[expansion=false]{microtype}

\makeatletter
\let\old@footnotetext\@footnotetext 
\newcommand{\subjclass}[2][2000]{%
  \let\@thefnmark\relax
  \old@footnotetext{\hskip-1.5em\textit{#1 Mathematical Subject Classification:} #2}%
}
\newcommand{\keywords}[1]{%
  \let\@thefnmark\relax
  \old@footnotetext{\hskip-1.5em\textit{Keywords:} #1}%
}
\makeatother

\numberwithin{equation}{section}

\usepackage[pdfstartview=XYZ]{hyperref}

\newtheorem{theorem}{Theorem}[section]
\newtheorem{lemma}[theorem]{Lemma}
\newtheorem{proposition}[theorem]{Proposition}
\newtheorem{corollary}[theorem]{Corollary}

\theoremstyle{definition}

\newtheorem{example}[theorem]{Example}
\newtheorem{examples}[theorem]{Examples}

\theoremstyle{remark}
\newtheorem{remark}[theorem]{Remark}
\newtheorem{remarks}[theorem]{Remarks}

\newcounter{subenv}[theorem]

\newenvironment{subremarks}{\begin{list}
  {\textit{\alph{subenv}}\rparen}
  {\leftmargin=1cm
   \usecounter{subenv}
   \def\makelabel##1{\hss\llap{##1}}}}{\end{list}}
\newenvironment{subexamples}{\begin{list}
  {\textit{\alph{subenv}}\rparen}
  {\leftmargin=1cm
   \usecounter{subenv}
   \def\makelabel##1{\hss\llap{##1}}}}{\end{list}}
\newenvironment{subproperty}{\begin{list}
  {\textit{\alph{subenv}}\textup{\rparen}}
  {\leftmargin=1cm
   \usecounter{subenv}
   \def\makelabel##1{\hss\llap{##1}}}}{\end{list}}

\newcommand*{\phantombigstrut}{\displaystyle\vphantom{\frac{(a)}{(b)}}}

\providecommand*{\rparen}{)}

\newcommand{\nth}[1]{#1\textsuperscript{th}}

\let\emptyset\varnothing

\newcommand{\namedash}{--}
\newcommand{\rangedash}{--}

\newcommand*{\card}[1]{\#{#1}}
\newcommand*{\set}[1]{\{#1\}}
\newcommand*{\setst}[2]{\{#1\mid#2\}}
\newcommand*{\diffens}[2]{#1\setminus#2}
\newcommand*{\img}{\mathop{\mathrm{Im}}}
\newcommand*{\iso}{\simeq}
\newcommand*{\onto}{\twoheadrightarrow}
\newcommand*{\isoto}{\overset{\sim}{\to}}
\newcommand*{\into}{\hookrightarrow}
\newcommand*{\xinto}[2][]{\xhookrightarrow[#1]{\mskip15mu#2\mskip15mu}}
\newcommand{\fullapp}[5]{\arraycolsep=1.4pt\begin{array}{ccccl}#1 & {}:{} & #2 & \to & #3\\[3pt] & & #4 & \mapsto & {\displaystyle #5}\end{array}}
\newcommand{\fulliso}[5]{\arraycolsep=1.4pt\begin{array}{lccl}#1\colon& \!#2 & \isoto & #3\\[3pt] & \!#4 & \mapsto & {\displaystyle #5}\end{array}}

\newcommand*{\abs}[1]{\lvert#1\rvert}

\newcommand*{\binomial}[2]{\binom{#1}{#2}}
\newcommand*{\integerinterval}[2]{[\![#1;#2]\!]}
\newcommand*{\divides}{\mid}
\newcommand*{\eulerphi}{\phi}

\newcommand*{\tensor}[1][]{\otimes_{#1}}
\newenvironment{system}{\begin{cases}}{\end{cases}}

\newcommand*{\ffield}[1]{\mathbb{F}_{\!#1}}
\newcommand*{\Fq}{\ffield{q}}
\newcommand*{\Fqnonzero}{\Fq^*}
\newcommand*{\Fqr}{\ffield{q^r}}
\newcommand*{\Fqbar}{\overline{\mathbb{F}}_{\!q}}
\newcommand*{\Fpbar}{\overline{\mathbb{F}}_{\!p}}

\newcommand*{\C}{\mathbb{C}}
\newcommand*{\K}{\mathbb{K}}
\newcommand*{\fieldk}{\Bbbk}
\newcommand{\NormFext}[2]{N_{#1/#2}}
\newcommand*\Tr{\mathop{\mathrm{Tr}}\nolimits}
\newcommand{\TrFext}[2]{\Tr_{#1/#2}}

\newcommand*{\N}{\mathbb{N}}
\newcommand*{\Z}{\mathbb{Z}}
\newcommand*{\ZnZ}{\Z/n\Z}
\newcommand*{\ZnZinv}[1][n]{(\Z/#1\Z)^\times}
\newcommand*{\ZnZnonzero}{\diffens{(\Z/n\Z)}{\set{0}}}
\newcommand*{\ZmZ}{\Z/m\Z}

\newcommand*{\Q}{\mathbb{Q}}
\newcommand*{\Qell}{\mathbb{Q}_\ell}
\newcommand*{\Qellbar}{\overline{\mathbb{Q}}_\ell}

\newcommand*{\unityroots}[1]{\boldsymbol{\mu}_{#1}}

\newcommand{\Id}{\mathop{\mathrm{Id}}\mathopen{}\mathord{}}
\newcommand{\Gal}{\mathop{\mathrm{Gal}}\nolimits}

\newcommand*{\varproj}[1]{\mathbb{P}^{#1}}
\newcommand*{\varprojFq}[1]{\varproj{#1}_{\scriptscriptstyle\!\mathbb{F}_{\mkern-1.5mu q}}}
\newcommand*{\varprojFpbar}[1]{\varproj{#1}_{\scriptscriptstyle\!\overline{\mathbb{F}}_{\mkern-1.5mu p}}}
\newcommand*{\varprojFqbar}[1]{\varproj{#1}_{\scriptscriptstyle\!\overline{\mathbb{F}}_{\mkern-1.5mu q}}}
\newcommand*{\colonsep}{\,\mathclose{:}\mathopen{}\,}
\newcommand{\PGL}[3][]{PGL_{#2}^{#1}(#3)}

\newcommand*{\parameter}{\psi}
\newcommand*{\dworkhypersurface}{X_\parameter}
\newcommand*{\dworkhypersurfaceFqbar}{\mkern2mu\overline{\mkern-2muX}_\parameter}
\newcommand*{\funzeta}[2]{Z_{#1}(#2)}
\newcommand*{\Xbar}{\overline{X}}

\newcommand*{\eulerpoincarecaract}{\chi}
\newcommand*{\trace}{\mathop{\mathrm{tr}}}
\newcommand*{\traceh}[2]{\trace(#1|\penalty600#2)}
\newcommand*{\deth}[2]{\det(#1|\penalty600#2)}
\newcommand*\Hetlad[3]{H_\textup{et}^{#1}(#2,\allowbreak#3)}
\newcommand*{\HetladQell}[2]{\Hetlad{#1}{#2}{\Qell}}
\newcommand*{\HetladQellbar}[2]{\Hetlad{#1}{#2}{\Qellbar}}
\newcommand*{\HetladQellprim}[2]{\Hetlad{#1}{#2}{\Qell}^\textup{prim}}
\newcommand*{\HetladQellbarprim}[2]{\Hetlad{#1}{#2}{\Qellbar}^\textup{prim}}
\newcommand*{\HetladQellinprim}[2]{\Hetlad{#1}{#2}{\Qell}^\textup{inprim}}

\newcommand*{\Waomega}[1][\omega]{W_{a,#1}}
\newcommand*{\Vaomega}[1][\omega]{V_{a,#1}}
\newcommand*{\Waomegaprime}{W_{a',\omega '}}
\newcommand*{\Waomegaeta}{W_{a,\omegaeta}}
\newcommand{\Waprimeomegaetaprime}{W_{a',\omega(\etaa ')}}
\newcommand*{\D}{D}
\newcommand*{\Da}{\D_{a}}
\newcommand*\Frob{\mathop{\mathrm{Frob}}\mathopen{}\mathord{}}
\newcommand*{\Frobcohom}{\Frob^*}
\newcommand*{\isotypiccomponentQellbar}[1]{\mkern2.5mu\overline{\mkern-2.5mu H}_{#1}}
\newcommand*{\isotypiccomponentQell}[2][\omega]{H_{\classZnZinv{#2},#1}}
\newcommand*{\Qa}[1][\omega]{Q_{a,#1}}
\newcommand*{\Qaomegaeta}{Q_{a,\omegaeta}}
\newcommand{\Qaex}[2]{Q_{#1,#2}}
\newcommand{\Pa}{P_{a,\omega}}
\newcommand{\Paex}[2]{P_{#1,#2}}
\newcommand{\Ra}{R_{a}}
\newcommand{\Raex}[1]{R_{#1}}

\newcommand*{\groupG}{G}
\newcommand*{\groupA}{A}
\newcommand*{\cargroupA}{\hat{\groupA}}
\newcommand*{\cargroupAtau}{\smash{\hat{\groupA}}{\mathchoice{\vrule width 0pt height 0.75em}{\vrule width 0pt height 0.75em}{\vrule width 0pt height 0.525em}{\vrule width 0pt height 0.375em}}^{\tau}}
\newcommand*{\cargroupAsigma}{\smash{\hat{\groupA}}{\mathchoice{\vrule width 0pt height 0.75em}{\vrule width 0pt height 0.75em}{\vrule width 0pt height 0.525em}{\vrule width 0pt height 0.375em}}^{\sigma}}
\newcommand*{\symmetricgroup}{\mathfrak{S}}
\newcommand*{\signature}{\epsilon}
\newcommand*{\Sn}{\symmetricgroup_n}
\newcommand*{\Hom}{\mathop{\mathrm{Hom}}\nolimits}
\newcommand*{\End}{\mathop{\mathrm{End}}\nolimits}
\newcommand*{\Aut}{\mathop{\mathrm{Aut}}\nolimits}
\newcommand*{\Ind}{\mathop{\mathrm{Ind}}\nolimits}
\newcommand*{\gr}[1]{\langle#1\rangle}
\newcommand*{\reg}{\mathop{\mathrm{reg}}\mathopen{}\mathord{}}
\newcommand*{\car}{\chi}
\newcommand{\carbis}{\xi}
\newcommand{\fieldcyclgroup}[1]{\K_{#1}}
\newcommand{\carcyclgroup}[1]{\car_{#1}}
\newcommand{\fieldcyclgroupZnZinv}[1]{\K_{#1}}
\newcommand{\carcyclgroupZnZinv}[1]{\car_{#1}}

\newcommand{\Ka}[1][a]{\K_{#1}}
\newcommand{\cara}[1][]{\carcyclgroupZnZinv{#1a}}
\newcommand{\Kprimea}[1][a]{\K_{#1}'}

\newcommand{\classZnZinv}[1]{\bar{#1}}

\newcommand{\classSn}[1]{\langle #1 \rangle}
\newcommand*{\stabilizerGQellbar}[1]{\groupG_{#1}}

\newcommand{\SaQellbar}[1][a]{S_{#1}}
\newcommand{\SaQellbarprime}[1][a]{S_{#1}'}
\newcommand{\SaQell}[1][a]{S_{\classZnZinv{#1}}}
\newcommand{\SaprimeQell}{S_{\classZnZinv{a}{}'}}
\newcommand{\SigmaQellbar}[1][a]{\overline{\Sigma}_{#1}}
\newcommand{\genSigmaQellbar}{\sigma}

\newcommand*{\ga}{{}^{g\mkern-1mu}a}
\newcommand*{\sigmaa}{{}^{\sigma\!}a}
\newcommand*{\taua}{{}^{\tau\!}a}

\newcommand*{\dpart}[2]{\frac{\partial#1}{\partial#2}}

\newcommand{\ma}[1][a]{m_{#1}}
\newcommand{\mprimea}[1][a]{m'_{#1}}
\newcommand{\mprimeaprime}[1][a]{m'_{#1 '}}
\newcommand{\nprimea}[1][a]{n'_{#1}}
\newcommand{\da}[1][a]{d_{#1}}

\newcommand{\ja}[1][a]{j_{#1}}
\newcommand{\gammaa}{\gamma_a}
\newcommand{\na}{n_{a}}
\newcommand{\fa}{f_{a}}
\newcommand{\ea}{e_{a}}
\newcommand{\eaprime}{e_{a'}}
\newcommand{\kerclassZnZinv}[1]{N_{#1}}
\newcommand{\imgclassZnZinv}[1]{E_{#1}}
\newcommand{\ka}[1][a]{k_{#1}}
\newcommand{\usigma}{u_\sigma}
\newcommand{\vsigma}{v_\sigma}
\newcommand{\usigmaprime}{u_{\sigma '}}
\newcommand{\vsigmaprime}{v_{\sigma '}}
\newcommand{\us}{u_s}
\newcommand{\ua}{u_a}
\newcommand{\uka}{u_{ka}}
\newcommand{\uprimea}{u_a'}
\newcommand{\mua}[1][\omega]{\mu_{a,#1}}
\newcommand{\Mua}[1][\omega]{M_{a,#1}}
\newcommand{\Muaprime}[1][\omega ']{M_{a'\!,\mkern1mu#1}}
\newcommand{\iotaa}{\iota_a}
\newcommand{\etaa}{\eta}
\newcommand{\omegaeta}{\omega(\etaa)}
\newcommand{\Muaomegaeta}{M_{a,\omegaeta}}
\newcommand{\va}{v_{a,\omega}}

\newcommand{\ZnaZ}{\Z/\na\Z}
\newcommand{\ZnaZinv}{(\ZnaZ)^\times}
\newcommand{\ZeaZ}{\Z/\ea\Z}
\newcommand{\ZeaZinv}{(\ZeaZ)^\times}

\newcommand{\prescriptzeta}[1]{\mathinner{{}^{#1\!}\zeta}\mathclose{}\mathord{}}

\newcommand*{\sctn}[1]{\S#1}
\newcommand*{\subsctn}[1]{\S#1}
\newcommand*{\subsctns}[1]{\S\S#1}
\newcommand*{\tbl}[1]{Table~#1}
\newcommand*{\frml}[1]{Formula~#1}
\newcommand*{\pg}[1]{page~#1}
\newcommand*{\pgs}[2]{pages~#1\rangedash#2}
\newcommand*{\prpstn}[1]{Proposition~#1}
\newcommand*{\thrm}[1]{Theorem~#1}
\newcommand*{\crllr}[1]{Corollary~#1}
\newcommand*{\rmrk}[1]{Remark~#1}
\newcommand*{\rmrks}[1]{Remarks~#1}
\newcommand*{\lmm}[1]{Lemma~#1}
\newcommand*{\exmpl}[1]{Example~#1}
\newcommand*{\exps}[1]{expos\'e~#1}

\title{Isotypic Decomposition of the Cohomology and Factorization of the Zeta Functions of Dwork Hypersurfaces}
\author{Philippe Goutet}

\begin{document}

\maketitle

\begin{abstract}
The aim of this article is to illustrate, on the Dwork hypersurfaces $x_1^n + \dots + x_n^n - n \parameter x_1 \dots x_n = 0$ (with $n$ an integer $\geq 3$ and $\parameter \in \Fqnonzero$ a parameter satisfying $\parameter^n \neq 1$), how the study of the representation of a finite group of automorphisms of a hypersurface in its etale cohomology allows to factor its zeta function.
\end{abstract}


\section{Introduction}\label{section:introduction}

Let $n$ be an integer $\geq 3$ and $\Fq$ a finite field of characteristic $p \neq 2$ not dividing~$n$; to simplify the results, we will assume that $q \equiv 1 \mod n$. We consider the projective hypersurface $\dworkhypersurface \subset \varprojFq{n-1}$ given by
\[x_1^n + \dots + x_n^n - n \parameter x_1 \dots x_n = 0,\]
where $\parameter$\label{definition:parameter} is a non zero parameter belonging to $\Fq$. The zeta function of $\dworkhypersurface$ is defined as
\[\funzeta{\dworkhypersurface/\Fq}{t} = \exp\biggl(\sum_{r=1}^{+\infty}{\card{\dworkhypersurface(\Fqr)} \frac{t^r}{r}}\biggr).\]
We assume that $\parameter^n \neq 1$, so that $\dworkhypersurfaceFqbar = \dworkhypersurface \tensor[\Fq] \Fqbar$ is nonsingular. As $\dworkhypersurface$ is a non-singular hypersurface of $\varproj{n-1}$, we know that the dimension of the etale $\ell$-adic cohomology spaces $\HetladQell{i}{\dworkhypersurfaceFqbar}$ is zero for $i > 2n-4$ or $i<0$ and that, for $0 \leq i \leq 2n-4$,
\[\dim \HetladQell{i}{\dworkhypersurfaceFqbar} = \begin{cases} \delta_i & \text{if 
$i \neq n-2$,} \\ 
\delta_i + \frac{(n-1)^n+(-1)^n(n-1)}{n} & \text{if $i = n-2$,}\end{cases}\]
where $\delta_i = 0$\label{definition:delta.i} if $i$ is odd and $\delta_i = 1$ if $i$ is even (see \subsctn{\ref{subsection:preliminaries:euler.poincare}}). As we will recall in \rmrk{\ref{remark:hetprim}} page~\pageref{remark:hetprim}, the zeta function of $\dworkhypersurface$ is related to how the Frobenius acts on $\HetladQell{n-2}{\dworkhypersurfaceFqbar}$.

We set\label{definition:groupA}\label{definition:cargroupA}
\begin{align*}
& \groupA = \setst{(\zeta_1,\dots,\zeta_n) \in \unityroots{n}(\Fq)^n}{\zeta_1\dots\zeta_n=1}/\set{(\zeta,\dots,\zeta)}; \\
& \cargroupA = \setst{(a_1,\dots,a_n) \in (\ZnZ)^n}{a_1+\dots+a_n=0}/\set{(a,\dots,a)},
\end{align*}
and denote by $[\zeta_1,\dots,\zeta_n]$\label{definition:elt.groupA} the class of $(\zeta_1,\dots,\zeta_n)$ in $\groupA$ and $[a_1,\dots,a_n]$\label{definition:elt.cargroupA} that of $(a_1,\dots,a_n)$ in $\cargroupA$. We will identify the group $\cargroupA$ with the group of characters of $\groupA$ taking values in $\Fqnonzero$. The group $\groupA$ acts on $\dworkhypersurface$ by coordinatewise multiplication; the symmetric group $\Sn$ acts on the right on $\dworkhypersurface$ by permutation of the coordinates
\[[x_1 \colonsep \dots \colonsep x_n]^{\sigma} = [x_{\sigma(1)} \colonsep \dots \colonsep x_{\sigma(n)}],\]
and on the left on $\groupA$ and $\cargroupA$ by
\begin{align*}
& {}^\sigma[\zeta_1,\dots,\zeta_n] = [\zeta_{\sigma^{-1}(1)},\dots,\zeta_{\sigma^{-1}(n)}]; \\
& {}^\sigma[a_1,\dots,a_n] = [a_{\sigma^{-1}(1)},\dots,a_{\sigma^{-1}(n)}].
\end{align*}
The semidirect product $\groupG = \groupA \rtimes \Sn$\label{definition:groupG} acts on the right on $\dworkhypersurface$, and hence on the left on $\HetladQell{n-2}{\dworkhypersurfaceFqbar}$ as the functor $g \mapsto g^*$ is contravariant.

The aim of this article is to describe the structure of $\HetladQell{n-2}{\dworkhypersurface}$ as a $\Qell[\groupG]$-module in order to deduce a factorization of the zeta function of $\dworkhypersurface$. More precisely, we will show that the primitive part of $\HetladQell{n-2}{\dworkhypersurface}$ (as defined in \subsctn{\ref{subsection:preliminaries:euler.poincare}}) admits an isotypic decomposition
\[\bigoplus_{a,\omega}{\Waomega \tensor[\Da] \Vaomega},\]
where $a$ describes $(\Sn \times \ZnZinv) \backslash \cargroupA$, $\omega$ belongs to a certain set of roots of unity (see \crllr{\ref{result:decomp:Ma.Qellbar}} page~\pageref{result:decomp:Ma.Qellbar}), $\Waomega$ is a simple $\Q[\groupG]$-module which is independent of $\ell$, $\Da$ is the division ring $\End_{\Q[\groupG]}(\Waomega)^\textup{opp}$, and $\Vaomega$ is a free module over $\Da \tensor[\Q] \Qell$ whose rank is independent of $\ell$. Because the Frobenius stabilizes these isotypic spaces, its characteristic polynomial splits in as many factors (the idea to use this method is inspired by an argument given in \cite[\subsctn{6.2}]{HKS}).

The first step is to decompose the $\Qellbar[G]$-module $\HetladQellbar{n-2}{\dworkhypersurfaceFqbar}$; we follow the same method Br\"unjes used for the case $\parameter = 0$ (Fermat hypersurface), but, thanks to a more powerful trace formula, we avoid the tedious induction of \cite[\prpstn{11.5}]{Brunjes}. Our methods can be generalized to other families of hypersurfaces, allow us to obtain factorizations slightly finer than those of \cite{Kloosterman} (who uses the $p$-adic Monsky-Washnitzer cohomology), and also allow us to express each factor as the norm of a polynomial with coefficients in a certain finite extension of $\Q$, hence explaining a numerical observation of Candelas, de la Ossa and Rodriguez-Villegas in the case $n=5$ where this extension is $\Q(\sqrt{5})$ (see \cite[\tbl{12.1} \pg{133}]{CdlORV.II}\footnote{They make this observation only in the case $\parameter=0$, but their numerical data in \sctn{13.3} suggests the same phenomenon happens when $\parameter \neq 0$ and $q \equiv 1 \mod 5$.}). Let us also mention that, in a recent article, \cite{Katz.another.look} studies the action of $\groupA$ (but not of $\groupA \rtimes \Sn$) on the cohomology of $\dworkhypersurface$ and establishes a motivic link between $\dworkhypersurface$ and objects of hypergeometric type.

The article is organized as follows. After preliminaries (\sctn{\ref{section:preliminaries}}), we describe the structure of $\HetladQellbar{n-2}{\dworkhypersurfaceFqbar}$ as a $\Qellbar[\groupA]$-module (\sctn{\ref{section:Qellbar[A]}}) and then as a $\Qellbar[\groupG]$-module (\sctn{\ref{section:Qellbar[G]}}). We then deduce the structure of the $\Qell[\groupG]$-module $\HetladQell{n-2}{\dworkhypersurfaceFqbar}$ (\sctn{\ref{section:Qell[G]}}) and explain the link between this structure and the existence of a factorisation of the zeta function of $\dworkhypersurface$ (\sctn{\ref{section:fact.zeta}}). An index of all notations introduced in the article is given in \sctn{\ref{appendix:notations}} and a table of the main formulas appears in \sctn{\ref{appendix:formulas}}.

\section{Preliminaries}\label{section:preliminaries}

We begin by recalling a Lefschetz-type trace formula by Deligne and Lusztig which allows to express the alternating sum of the traces of an automorphism on the $\ell$-adic cohomology spaces as the Euler\namedash Poincar\'e characteristic of the fixed-point scheme of this automorphism. We then recall the value of this Euler\namedash Poincar\'e characteristic in the cases we will encounter in what follows (smooth projective hypersurfaces). Finally, we link the trace of an element of $\groupG$ to the Euler\namedash Poincar\'e characteristic of a subscheme of fixed points.

\subsection{Lefschetz trace formula}\label{subsection:preliminaries:lefschetz}

Let us recall that the Euler\namedash Poincar\'e characteristic\label{definition:euler-poincare.caract} of a proper scheme over $\Fpbar$ is given by
\[\eulerpoincarecaract(X) = \sum_{i=0}^{2\dim X}{(-1)^i \dim \HetladQell{i}{X}},\]
where $\ell$ is a prime number $\neq p$. It is an integer independent of $\ell$.

\begin{theorem}\label{theorem:trace.lefschetz:EP}
Let $X$ be a proper scheme over $\Fpbar$. If $f$ is an automorphism of $X$ of finite order prime to $p$, and if $X^f$\label{definition:Xf} denotes the fixed-point subscheme of $f$ of the scheme $X$, then
\[\sum_{i=0}^{2\dim X}{(-1)^i\traceh{f^*}{\HetladQell{i}{X}}} = \eulerpoincarecaract(X^f).\]
\end{theorem}

\begin{proof}
See \cite[\thrm{3.2}, \pg{119}]{Deligne.Lusztig}.
\end{proof}

\subsection{Euler\texorpdfstring{\namedash}{-}Poincar\'e characteristic of a non-singular hypersurface}\label{subsection:preliminaries:euler.poincare}

In this \subsctn{\ref{subsection:preliminaries:euler.poincare}}, exceptionally, we do not assume that $n \geq 3$.

\begin{theorem}[Hirzebruch formula]\label{theorem:hirzebruch}
Let $n$ be an integer $\geq 1$ and $f \in \Fpbar[x_1,\dots,\allowbreak x_n]$ a homogeneous polynomial of degree $d$ such that $f$, $\dpart{f}{x_1}$, \dots, $\dpart{f}{x_n}$ have no common zero in $\Fpbar^n$ except $(0,\dots,0)$. Then the hypersurface $X \subset \varprojFpbar{n-1}$ defined by $f = 0$ is non-singular (and irreducible if $n \geq 3$) and its Euler\namedash Poincar\'e characteristic is
\[\eulerpoincarecaract(X) = (n-1) + \frac{(1-d)^n+(d-1)}{d}.\]
\end{theorem}

\begin{proof}
If $n \geq 3$, we use \crllr{7.5.$(iii)$} of \cite[\exps{VII}]{SGA5}: indeed, the subscheme $X$ of $\varprojFpbar{n-1}$ is smooth, connected and of dimension $n-2$; its Euler\namedash Poincar\'e characteristic is hence
\begin{align*}
\eulerpoincarecaract(X)
& = d \sum_{i=0}^{n-2}{(-1)^{n-i}\binomial{n}{i} d^{n-2-i}} = \frac{1}{d} \sum_{i=0}^{n-2}{(-1)^{n-i}\binomial{n}{i} d^{n-i}} \\
& = \frac{(1-d)^n + nd - 1}{d},
\end{align*}
which is the announced formula. If $n=2$, the hypersurface $X$ of $\varprojFpbar{1}$ consists of $d$ distinct points and so $\eulerpoincarecaract(X) = d$, which shows the result as $(2-1) + \frac{1}{d}[(1-d)^2+(d-1)] = d$. Finally, if $n = 1$, $X = \emptyset$ and so $\eulerpoincarecaract(X) = 0$, which also shows the result in this case.
\end{proof}

\begin{remark}\label{remark:hetprim}
When $n \geq 3$, \thrm{\ref{theorem:hirzebruch}} can be refined as follows. We keep the same notations and denote by $j$ the canonical injection $X \to \varprojFpbar{n-1}$. By the Weak Lefschetz Theorem, (see for example \cite[\crllr{9.4}, \pg{106}]{Freitag.Kiehl}), for $i < n-2$ (respectively $i=n-2$), the linear map $j^* \colon \HetladQell{i}{\varprojFpbar{n-1}} \to \HetladQell{i}{X}$ is bijective (respectively injective). If we set $\delta_i = 0$ if $i$ odd and $\delta_i = 1$ if $i$ is even, we thus have $\dim \HetladQell{i}{X} = \delta_i$ for $i < n-2$, and this result stays valid for $n-2 < i \leq 2(n-2)$ by Poincar\'e duality. For $i = n-2$, the image of the map $j^* \colon \HetladQell{n-2}{\varprojFpbar{n-1}} \to \HetladQell{n-2}{X}$ has dimension $\delta_i$. We will denote it by $\HetladQellinprim{n-2}{X}$\label{definition:cohom.et.inprim} and set $\HetladQellprim{n-2}{X} = {\HetladQell{n-2}{X}} / {\HetladQellinprim{n-2}{X}}$\label{definition:cohom.et.prim}. Because the Frobenius acts as the multiplication by $q^{(n-2)/2}$ on $\HetladQellinprim{n-2}{\dworkhypersurfaceFqbar}$ and by multiplication by $q^i$ on each $\HetladQell{2i}{\dworkhypersurfaceFqbar}$, we have
\[\funzeta{\dworkhypersurface/\Fq}{t} = \frac{\deth{1-t\Frobcohom}{\HetladQellprim{n-2}{\dworkhypersurfaceFqbar}}^{(-1)^{n-1}}}{(1-t)(1-qt)\dots(1-q^{n-2}t)}.\]
\end{remark}

\subsection{Character of \texorpdfstring{$G$}{G} acting on \texorpdfstring{$\HetladQellprim{n-2}{\dworkhypersurfaceFqbar}$}{Hn-2(X,Ql)prim}}\label{subsection:preliminaries:character.values}

The isomorphism class of a $\Qell[\groupG]$-module is completely determined by its character. In this~\subsctn{\ref{subsection:preliminaries:character.values}}, we will express in terms of Euler\namedash Poincar\'e characteristics the values of the character of the $\Qell[\groupG]$-module $\HetladQellprim{n-2}{\dworkhypersurfaceFqbar}$ for the elements $g \in \groupG$ which are of order prime to $p$.

\begin{lemma}\label{result:trivial.action.Hi}
Each $g \in \groupG$ acts as the identity on $\HetladQellinprim{n-2}{\dworkhypersurfaceFqbar}$ and on $\HetladQell{i}{\dworkhypersurfaceFqbar}$ when $i \neq n-2$.
\end{lemma}

\begin{proof}
As $g$ is the restriction of an automorphism of $\varprojFqbar{n-1}$, it results from \rmrk{\ref{remark:hetprim}} and the following lemma.
\end{proof}

\begin{lemma}\label{lemme:h*surPn}
If $h$ is an automorphism of $\varprojFqbar{n-1}$, then $h^*$ acts as the identity on $\HetladQell{i}{\varprojFqbar{n-1}}$ for all $i$.
\end{lemma}

\begin{proof}
The group $\PGL{n}{\Fqbar}$ acts on the right on $\HetladQell{i}{\varprojFqbar{n-1}}$ by $u \mapsto u^*$; as $\HetladQell{i}{\varprojFqbar{n-1}}$ is of dimension $0$ or $1$, this action is by homothety, and thus factors by an abelian quotient of $\PGL{n}{\Fqbar}$. Since $\Fqbar$ is algebraically closed, $\PGL{n}{\Fqbar}$ is equal to its commutator subgroup and thus has no nonzero abelian quotient. Hence, for all $u \in \PGL{n}{\Fqbar}$, $u^* = \Id$.
\end{proof}

\begin{theorem}\label{result:generic.formula.trace.G}
If $g \in G$ is of order prime to $p$, then
\begin{equation}\label{formula:lefschetz.trace:EP}
\traceh{g^*}{\HetladQellprim{n-2}{\dworkhypersurfaceFqbar}} = (-1)^{n-1}\Big((n-1) - \eulerpoincarecaract(\dworkhypersurfaceFqbar^g)\Big).
\end{equation}
\end{theorem}

\begin{proof}
Using the trace formula of \thrm{\ref{theorem:trace.lefschetz:EP}}, we can write
\[\sum_{i=0}^{2\dim X}{(-1)^i\traceh{g^*}{\HetladQell{i}{\dworkhypersurfaceFqbar}}} = \eulerpoincarecaract(\dworkhypersurfaceFqbar^g).\]
By \lmm{\ref{result:trivial.action.Hi}}, we have (with, as previously, $\delta_i = 0$ if $i$ is odd and $\delta_i = 1$ if $i$ is even)
\[\traceh{g^*}{\HetladQell{i}{\dworkhypersurfaceFqbar}} = \begin{cases} \delta_i & \text{if $i \neq n-2$,} \\ \delta_i + \traceh{g^*}{\HetladQellprim{i}{\dworkhypersurfaceFqbar}} & \text{if $i=n-2$,}\end{cases}\]
and thus
\[\eulerpoincarecaract(\dworkhypersurfaceFqbar^g) = (n-1) + (-1)^{n-2} \traceh{g^*}{\HetladQellprim{n-2}{\dworkhypersurfaceFqbar}},\]
which is exactly the announced formula.
\end{proof}

\section{Action of \texorpdfstring{$\groupA$}{A} on \texorpdfstring{$\HetladQellbarprim{n-2}{\dworkhypersurfaceFqbar}$}{Hn-2(X,Qlbar)prim}}\label{section:Qellbar[A]}

The irreducible representations over $\Qellbar$ of the finite abelian group $\groupA$ are its characters (of degree~$1$). Finding the structure of the $\Qellbar[\groupA]$-module $\HetladQellbar{n-2}{\dworkhypersurfaceFqbar}$ hence amounts to figuring out the multiplicity of each character of $\groupA$ in the representation $g \mapsto g^*$ of $\groupA$ in $\HetladQellbar{n-2}{\dworkhypersurfaceFqbar}$; it is the aim of this \sctn{\ref{section:Qellbar[A]}}. The choice, in \subsctn{\ref{subsection:Qellbar[A]:cargroup.A.Qellbar}}, of an isomorphism between $\unityroots{n}(\Fq)$ and $\unityroots{n}(\Qellbar)$ allows to identify $\cargroupA$ to the group of characters of $\groupA$ taking values in $\Qellbar$. After determining the character of the $\Qellbar[\groupA]$-module $\HetladQellbarprim{n-2}{\dworkhypersurfaceFqbar}$ in \subsctn{\ref{subsection:Qellbar[A]:character.values.A}}, we will prove in \subsctn{\ref{subsection:decomp.repr.irr.A}} that the multiplicity of $a \in \cargroupA$ is $\ma = \card{(\diffens{\ZnZ}{\set{a_1,\dots,a_n}})}$.

\subsection{Characters of \texorpdfstring{$\groupA$}{A} with values in \texorpdfstring{$\Qellbar$}{Qlbar}}\label{subsection:Qellbar[A]:cargroup.A.Qellbar}

As we only consider the case $q \equiv 1 \mod n$, the group $\unityroots{n}(\Fq)$ consisting of the \nth{$n$} roots of unity of $\Fq$ is isomorphic to the group of \nth{$n$} roots of unity of $\Qellbar$. We call $t$ an isomorphism of $\unityroots{n}(\Fq)$ onto $\unityroots{n}(\Qellbar)$ and use it to identify the group $\cargroupA$ with the group of characters of $\groupA$ taking values in $\Qellbar$ thanks to the isomorphism $[a_1,\dots,a_n] \mapsto ([\zeta_1,\dots,\zeta_n] \mapsto t(\zeta_1)^{a_1}\cdots t(\zeta_n)^{a_n})$.

\subsection{Character values of the \texorpdfstring{$\Qellbar[\groupA]$}{Qlbar[A]}-module \texorpdfstring{$\HetladQellbarprim{n-2}{\dworkhypersurfaceFqbar}$}{Hn-2(X,Qlbar)prim}}\label{subsection:Qellbar[A]:character.values.A}

As $p$ is prime to $n$ by assumption, the elements of $\groupA$ have an order prime to $p$; we may thus use \frml{\eqref{formula:lefschetz.trace:EP}} to obtain the values taken by the characters of the $\Qellbar[\groupA]$-module $\HetladQellbarprim{n-2}{\dworkhypersurfaceFqbar}$.

\begin{theorem}\label{result:trace.z1...zn}
Consider $(\zeta_1,\dots,\zeta_n) \in \unityroots{n}(\Fq)^n$ such that $\zeta_1\dots\zeta_n = 1$ and let $g$ be the corresponding element $[\zeta_1,\dots,\zeta_n]$ of $\groupA$. For all $\zeta \in \unityroots{n}(\Fq)$, denote by $k(\zeta)$\label{definition:kzeta} the number of $i \in \integerinterval{1}{n}$ such that $\zeta_i = \zeta$. We have
\begin{equation}\label{formula:trace.z1...zn}
\traceh{g^*}{\HetladQellbarprim{n-2}{\dworkhypersurfaceFqbar}} = \frac{(-1)^n}{n} \sum_{\zeta \in \unityroots{n}(\Fq)}{(1-n)^{k(\zeta)}}.
\end{equation}
\end{theorem}

\begin{proof}
A point of $\dworkhypersurfaceFqbar$ with homogeneous coordinates $[x_1 \colonsep \dots \colonsep x_n]$ is a fixed point of $g$ if and only if $(x_1,\dots,x_n)$ is proportional to $(\zeta_1x_1 , \dots,\zeta_nx_n)$. The proportionality coefficient is necessarily a root of unity $\zeta \in \unityroots{n}(\Fq)$, and we must have $x_i = 0$ if $\zeta_i \neq \zeta$. Hence, the subscheme of fixed points of $g$ of $\dworkhypersurfaceFqbar$ is the disjoint union over $\zeta \in \unityroots{n}(\Fq)$ of the subvarieties
\[Y_{\zeta} = \setst{x \in \dworkhypersurfaceFqbar}{x_i = 0 \text{ if $\zeta_i \neq \zeta$}}\text{.}\]
If $k(\zeta) = n$, we have $Y_{\zeta} = \dworkhypersurfaceFqbar$. If $2 \leq k(\zeta) \leq n-1$, $Y_{\zeta}$ is isomorphic to the hypersurface of $\varproj{k(\zeta)-1}$ defined by $y_1^n + y_2^n + \dots + y_{k(\zeta)}^n = 0$. Finally, if $k(\zeta) = 0$ or $1$, $Y_{\zeta}$ is empty. In each of these cases, we can apply \thrm{\ref{theorem:hirzebruch}} and obtain
\[\eulerpoincarecaract(Y_{\zeta}) = k(\zeta)-1 + \frac{(1-n)^{k(\zeta)}+n-1}{n} = k(\zeta) - \frac{1}{n} + \frac{(1-n)^{k(\zeta)}}{n}.\]
Consequently, since $\sum_{\zeta \in \unityroots{n}(\Fq)}{k(\zeta)} = n$ and $\sum_{\zeta \in \unityroots{n}(\Fq)}{\frac{1}{n}} = 1$,
\[\eulerpoincarecaract(\dworkhypersurfaceFqbar^g) = \sum_{\zeta \in \unityroots{n}(\Fq)}{\eulerpoincarecaract(Y_{\zeta})} = n-1 + \sum_{\zeta \in \unityroots{n}(\Fq)}{\frac{(1-n)^{k(\zeta)}}{n}}.\]
Using trace formula~\eqref{formula:lefschetz.trace:EP} page~\pageref{formula:lefschetz.trace:EP}, we deduce the announced result.
\end{proof}

\begin{remark}
A recent preprint proves, in a more general setting, formulas of the type given in \thrm{\ref{result:trace.z1...zn}} and \thrm{\ref{result:trace:d-cycles}} \pg{\pageref{result:trace:d-cycles}}; see \cite[\crllr{2.5}]{Chenevert.representations.cohomology.symmetries}.
\end{remark}

\subsection{Decomposition in irreducible representations}\label{subsection:decomp.repr.irr.A}

The following theorem gives a simple expression for the multiplicity $\ma$ of a character $a \in \cargroupA$ in the $\Qellbar[\groupA]$-module $\HetladQellbarprim{n-2}{\dworkhypersurfaceFqbar}$.

\begin{theorem}\label{result:mult.a}
The multiplicity of the irreducible character $a = [a_1,\dots,a_n]$ of $\groupA$ in the $\Qellbar[\groupA]$-module $\HetladQellbarprim{n-2}{\dworkhypersurfaceFqbar}$ is\label{definition:ma}
\[\ma = \card{(\diffens{\ZnZ}{\set{a_1,\dots,a_n}})} = n - (\text{number of distinct $a_i$}).\]
\end{theorem}

\begin{proof}
Consider $(\zeta_1, \dots, \zeta_n) \in \unityroots{n}(\Fq)^n$ such that $\zeta_1\dots\zeta_n = 1$ and let $g$ be the corresponding element $[\zeta_1,\dots,\zeta_n]$ of $\groupA$. From the definition, we have
\begin{align*}
\traceh{g^*}{\HetladQellbarprim{n-2}{\dworkhypersurfaceFqbar}}
& = \sum_{\substack{a\in\cargroupA}}{\ma \zeta_1^{a_1}\dots\zeta_n^{a_n}}\\
& = \frac{1}{n} \sum_{\substack{(a_1,\dots,a_n)\in(\ZnZ)^n \\ a_1+\dots+a_n=0}}{\ma \zeta_1^{a_1}\dots\zeta_n^{a_n}}.
\end{align*}
We will show that if we replace $\ma$ by the number of elements of $\diffens{\ZnZ}{\set{a_1,\dots,a_n}}$ in the right hand side, we recover \frml{\eqref{formula:trace.z1...zn}} above, which will show the announced result. We write
\begin{align*}
& \frac{1}{n} \sum_{\substack{(a_1,\dots,a_n)\in(\ZnZ)^n\\a_1+\dots+a_n=0}}{\card{(\diffens{\ZnZ}{\set{a_1,\dots,a_n}})} \zeta_1^{a_1}\dots\zeta_n^{a_n}} \\
& \qquad\qquad\qquad = \frac{1}{n} \sum_{\substack{(a_1,\dots,a_n)\in(\ZnZ)^n\\a_1+\dots+a_n=0}}{\bigg(\sum_{\substack{k \in \ZnZ \\ \forall i\text{, }a_i \neq k}}{1}\bigg)\zeta_1^{a_1}\dots\zeta_n^{a_n}} \\
& \qquad\qquad\qquad = \frac{1}{n} \sum_{k \in \ZnZ}{~\sum_{\substack{(a_1,\dots,a_n)\in(\ZnZ)^n\\a_1+\dots+a_n=0\\\forall i\text{, }a_i \neq k}}{\zeta_1^{a_1}\dots\zeta_n^{a_n}}} \\
& \qquad\qquad\qquad = \frac{1}{n} \sum_{k \in \ZnZ}{~\sum_{\substack{(a_1,\dots,a_n)\in(\ZnZ)^n\\a_1+\dots+a_n=0\\\forall i\text{, }a_i \neq 0}}{\zeta_1^{a_1}\dots\zeta_n^{a_n}}} \\
& \qquad\qquad\qquad = \sum_{\substack{(a_1,\dots,a_n)\in(\ZnZ)^n\\a_1+\dots+a_n=0\\\forall i\text{, }a_i \neq 0}}{\zeta_1^{a_1}\dots\zeta_n^{a_n}} \\
& \qquad\qquad\qquad = \sum_{\substack{a_1,\dots,a_n\in\ZnZnonzero\\a_1+\dots+a_n=0}}{\zeta_1^{a_1}\dots\zeta_n^{a_n}}.
\end{align*}
We now conclude by using the following lemma.
\end{proof}

\begin{lemma}\label{lemma:compute.sum.induction.A}
Let $r$ be an integer $\geq 1$ and $\zeta_1$, \dots, $\zeta_r$ elements of $\unityroots{n}(\Fq)$. If $k(\zeta) = k_{(\zeta_1,\dots,\zeta_r)}(\zeta)$ denotes the number of $i \in \integerinterval{1}{r}$ such that $\zeta_i = \zeta$, then
\[\sum_{\substack{a_1,\dots,a_r\in\ZnZnonzero \\ a_1+\dots+a_r=0}} {\zeta_1^{a_1}\dots\zeta_r^{a_r}} = \frac{(-1)^r}{n} \sum_{\zeta \in \unityroots{n}(\Fq)}{(1-n)^{k(\zeta)}}.\]
\end{lemma}

\begin{proof}
We proceed by induction on $r$. For $r=1$, the equality is the relation
\[0 = -\frac{1}{n}\Big((1-n)^1 + (n-1)(1-n)^0\Big).\]
We now assume that $r \geq 2$ and that the result is known for $r-1$. We write
\begin{align*}
\sum_{\substack{a_1,\dots,a_r\in\ZnZnonzero\\a_1+\dots+a_r=0}}{\zeta_1^{a_1}\dots\zeta_r^{a_r}} & = \sum_{\substack{a_1,\dots,a_{r-1}\in\ZnZnonzero\\a_1+\dots+a_{r-1}\neq 0}}{\,\zeta_1^{a_1}\dots\zeta_{r-1}^{a_{r-1}} \zeta_r^{-a_1-\dots-a_{r-1}}} \\
& = \sum_{a_1,\dots,a_{r-1}\in\ZnZnonzero}{\left(\frac{\zeta_1}{\zeta_r}\right)^{a_1} \dots \left(\frac{\zeta_{r-1}}{\zeta_r}\right)^{a_{r-1}} } \\
& \hspace{5em} - \sum_{\substack{a_1,\dots,a_{r-1}\in\ZnZnonzero\\a_1+\dots+a_{r-1}= 0}}{\,\zeta_1^{a_1}\dots\zeta_{r-1}^{a_{r-1}}}.
\end{align*}
Given $\zeta \in \unityroots{n}(\Fq)$, we have
\[\sum_{a \in \ZnZnonzero}{\zeta^a} = \begin{cases} -1 & \text{if $\zeta \neq 1$,} \\ n-1 & \text{if $\zeta = 1$.} \end{cases}\]
This little remark allows to compute the first sum:
\[\sum_{a_1,\dots,a_{r-1}\in\ZnZnonzero}{\left(\frac{\zeta_1}{\zeta_r}\right)^{a_1} \dots \left(\frac{\zeta_{r-1}}{\zeta_r}\right)^{a_{r-1}}} = (-1)^{r - k(\zeta_r)}(n-1)^{k(\zeta_r)-1},\]
where $k(\zeta) = k_{(\zeta_1,\dots,\zeta_r)}(\zeta)$. To compute the second sum, we use the induction assumption:
\[\sum_{\substack{a_1,\dots,a_{r-1}\in\ZnZnonzero\\a_1+\dots+a_{r-1}= 0}}{\zeta_1^{a_1}\dots\zeta_{r-1}^{a_{r-1}}} = \mathinner{\frac{(-1)^{r-1}}{n}} \biggl(\mathinner{}\sum_{\zeta \neq \zeta_r}{(1-n)^{k(\zeta)}}+ (1-n)^{k(\zeta_r)-1}\biggr).\]
We conclude by noting that
\begin{align*}
& (-1)^{r - k(\zeta_r)}(n-1)^{k(\zeta_r)-1} - \frac{(-1)^{r-1}}{n}(1-n)^{k(\zeta_r)-1} \\
& \qquad = -\frac{(-1)^r n(1-n)^{k(\zeta_r)-1}}{n} + \frac{(-1)^{r}}{n}(1-n)^{k(\zeta_r)-1} = \frac{(-1)^r}{n} (1-n)^{k(\zeta_r)}.\qedhere
\end{align*}
\end{proof}

\begin{remark}\label{remark:ma.non.zero}
As a consequence of \thrm{\ref{result:mult.a}}, the multiplicity $\ma$ of the character $a \in \cargroupA$ is nonzero unless $a$ belongs to the orbit of $[0,1,2,\dots,n-1]$ under $\symmetricgroup_n$ (which imposes $n$ odd, or else $1+2+\dots+(n-1)$ is not divisible by $n$).
\end{remark}

\section{Action of \texorpdfstring{$\groupG$}{G} on \texorpdfstring{$\HetladQellbarprim{n-2}{\dworkhypersurfaceFqbar}$}{Hn-2(X,Qlbar)prim}}\label{section:Qellbar[G]}

\subsection{A decomposition of the \texorpdfstring{$\Qellbar[\groupG]$}{Qlbar[G]}-module \texorpdfstring{$\HetladQellbarprim{n-2}{\dworkhypersurfaceFqbar}$}{Hn-2(X,Qlbar)prim}}\label{subsection:Qellbar[G]:isotypic.decomposition.Qellbar[G]}

For every $a$ belonging to $\cargroupA$ identified to the group of characters of $\groupA$ taking values in $\Qellbar$, we denote by $\isotypiccomponentQellbar{a}$\label{definition:composantisotypiqueQellbar.a} the isotypic component relatively to $a$ of the $\Qellbar[\groupA]$-module $\HetladQellbarprim{n-2}{\dworkhypersurfaceFqbar}$ (see \cite[\subsctn{3.4}]{Bourbaki.algebre.viii}). It is a $\Qellbar$-vector space of dimension $\ma$, where $\ma$ is the multiplicity computed in \subsctn{\ref{subsection:decomp.repr.irr.A}}, and we have
\[\HetladQellbarprim{n-2}{\dworkhypersurfaceFqbar} = \bigoplus_{a \in \cargroupA\mathstrut}{\isotypiccomponentQellbar{a}}.\]

The group $\groupG$ acts on the left on $\groupA$ by inner automorphisms, and thus acts on the left on $\cargroupA$: if $g \in \groupA \sigma$, with $\sigma \in \symmetricgroup_n$, and if $a = [a_1,\dots,a_n]$, we have $\ga = \sigmaa = [a_{\sigma^{-1}(1)},\dots,a_{\sigma^{-1}(n)}]$.

Consider $a \in \cargroupA$. Denote by $\classSn{a}$\label{definition:classeSna} the orbit of $a$ under $\symmetricgroup_n$. The stabilizer $\stabilizerGQellbar{a}$\label{definition:GaQellbar} of $a$ in $\groupG$ is equal to $\groupA \rtimes \SaQellbar$\label{definition:SaQellbar}, where $\SaQellbar = \setst{\sigma \in \symmetricgroup_n}{\sigmaa = a}$. We have $g \isotypiccomponentQellbar{a} = \isotypiccomponentQellbar{\ga}$ for all $g \in \groupG$ and the space $\isotypiccomponentQellbar{a}$ is stable by $\stabilizerGQellbar{a}$. The subspace $\bigoplus_{a' \in \classSn{a}}{\isotypiccomponentQellbar{a'}}$ of $\HetladQellbarprim{n-2}{\dworkhypersurfaceFqbar}$ is stable by $\groupG$; it is a $\Qellbar[\groupG]$-module canonically isomorphic to $\Ind_{\stabilizerGQellbar{a}}^{\groupG} \isotypiccomponentQellbar{a}$. We thus deduce the following result.

\begin{theorem}\label{result:isotypicA.decomposition.Qellbar[G]}
Denote by $R \subset \cargroupA$\label{definition:sysclasseSna} a set of representatives of $\Sn\backslash\cargroupA$. The $\Qellbar[\groupG]$-module $\HetladQellbarprim{n-2}{\dworkhypersurfaceFqbar}$ is isomorphic to
\[\bigoplus_{a\in R}\Ind_{\stabilizerGQellbar{a}}^{\groupG} \isotypiccomponentQellbar{a}.\]
\end{theorem}

The aim of the rest of this \sctn{\ref{section:Qellbar[G]}} is to determine how the group $\SaQellbar$ acts on $\isotypiccomponentQellbar{a}$. The strategy is the following: after showing that $\SaQellbar$ is a semi-direct product $\SaQellbarprime \rtimes \SigmaQellbar$ (\subsctn{\ref{subsection:Qellbar[G]:structure.SaQellbar}}), we compute $\traceh{\sigma^*}{\HetladQellbarprim{n-2}{\dworkhypersurfaceFqbar}}$ for $\sigma$ a generator of $\SaQellbarprime$ and compare it to the trace of the identity (\subsctn{\ref{subsection:Qellbar[G]:value.trace.identity}}) to deduce that $\SaQellbarprime$ acts as $\signature(\sigma)\Id_{\isotypiccomponentQellbar{a}}$ on $\isotypiccomponentQellbar{a}$ (see \subsctn{\ref{subsection:Qellbar[G]:action.Saprime}}). We then show, using a method similar to \sctn{\ref{section:Qellbar[A]}}, that $\SigmaQellbar$ acts as a multiple of the regular representation (\subsctn{\ref{subsection:Qellbar[G]:character.values.Sigmaa}\rangedash\ref{subsection:Qellbar[G]:action.Sigmaa}}).

The approach we use to study the action of $\SaQellbarprime$ is the same that Br\"unjes used in \cite[\prpstn{11.5}, \pg{197}]{Brunjes} for the case $\parameter = 0$, the only difference being that our trace formula allows us to avoid a tedious proof by induction.

\subsection{Structure of \texorpdfstring{$\SaQellbar$}{Sa}}\label{subsection:Qellbar[G]:structure.SaQellbar}

Consider $a = [a_1,\dots,a_n] \in \cargroupA$, where $(a_1,\dots,a_n)$ is an element of $(\ZnZ)^n$ such that $a_1 + \dots + a_n = 0$. The set of $j \in \ZnZ$ such that $(a_1+j,\dots,a_n+j)$ is a permutation of $(a_1,\dots,a_n)$ is a subgroup of $\ZnZ$; it can be written as $\nprimea\ZnZ$\label{definition:nprimea} for some integer $\nprimea \geq 1$ dividing $n$; let $\da = n/\nprimea$\label{definition:da} be the order of this group. These two integers only depend on $a$ and not on the choice of $a_1$, \dots, $a_n$.

\begin{remark}\label{remark:def.naprime}
For all $b \in \ZnZ$, denote by $I(b)$\label{definition:I(b)} the set of $i \in \set{1,\dots,n}$ such that $a_i = b$. The set $\nprimea \ZnZ$ is the set of $j \in \ZnZ$ such that $I(b+j)$ has the same number of elements as $I(b)$ for all $b \in \ZnZ$.
\end{remark}

\begin{lemma}
There is a permutation $\genSigmaQellbar \in \symmetricgroup_n$\label{definition:genSigmaQellbar} such that
\begin{subproperty}
    \item\label{condition:a} if $1 \leq i \leq n$, we have $a_{\genSigmaQellbar(i)} = a_i + \nprimea$~;
    \item\label{condition:b} $\genSigmaQellbar$ is the product of $\nprimea$ disjoint cycles of length $\da$.
\end{subproperty}
\end{lemma}

\begin{proof}
Let us note that the condition~\ref{condition:a} is equivalent to the fact that $\sigma(I(b)) = I(b+\nprimea)$. For all $b \in \ZnZ$ such that $I(b) \neq \emptyset$, choose a numbering $i_1(b)$, \dots, $i_{\card{I(b)}}(b)$ of the elements of $I(b)$ and denote by $\genSigmaQellbar$ the element of $\Sn$ which sends $i_l(b)$ to $i_l(b+\nprimea)$ for all $b \in \ZnZ$ and $1 \leq l \leq \card{I(b)}$.

From the definition, we have $a_{\genSigmaQellbar(i)} = a_i + \nprimea$ and, inspecting the orbits of each of the $a_i$ under $b \mapsto b + \nprimea$, we see that $\genSigmaQellbar$ is a product of $\nprimea$ disjoint cycles of length $\da$.
\end{proof}

Denote by $\SaQellbarprime$\label{definition:SaQellbarprime} the fixator of $(a_1,\dots,a_n) \in (\ZnZ)^n$ in $\symmetricgroup_n$; it is a group which can be identified with $\prod_{b \in \ZnZ}{\symmetricgroup_{I(b)}}$ (it is hence generated by transpositions) and we set $\gammaa = [\Sn:\SaQellbarprime]$\label{definition:gammaa}. Consider $\genSigmaQellbar \in \symmetricgroup_n$ satisfying the conditions of the preceding lemma and let $\smash{\SigmaQellbar} = \langle \genSigmaQellbar \rangle$\label{definition:SigmaQellbar} be the cyclic subgroup of order $\da$ of $\symmetricgroup_n$ generated by~$\genSigmaQellbar$.

\begin{proposition}\label{result:SaQellbar.semi-direct.prod}
The fixator $\SaQellbar$ of $a = [a_1,\dots,a_n] \in \cargroupA$ can be written as the semi-direct product
\[\SaQellbar = \SaQellbarprime \rtimes \SigmaQellbar.\]
\end{proposition}

\begin{proof}
If $s \in \SaQellbar$, there exists a unique $j \in \nprimea\ZnZ$ such that ${}^s (a_1,\dots,a_n) = (a_1+j,\dots,a_n+j)$. This element only depends on $a$, not on the choice of $a_1$, \dots, $a_n$; we denote it by $\ja(s)$. The map $\ja \colon \SaQellbar \to \nprimea\ZnZ$\label{definition:ja} thus defined is a group homomorphism. This homomorphism is surjective and its kernel is the fixator $\SaQellbarprime$ of $(a_1,\dots,a_n) \in (\ZnZ)^n$ in $\Sn$.

Moreover, as $a_{\genSigmaQellbar(i)} = a_i + \nprimea$ and thus $a_{\genSigmaQellbar^{-1}(i)} = a_i - \nprimea$, we have $\ja(\genSigmaQellbar) = -\nprimea$ by construction, hence $\ja$ induces an isomorphism of $\SigmaQellbar = \gr{\genSigmaQellbar}$ onto the image $\nprimea\ZnZ$ of $\ja$, which shows that
\[\SaQellbar = \SaQellbarprime \rtimes \SigmaQellbar.\qedhere\]
\end{proof}

\begin{remarks}\label{remark:da.nprimea:indt.a}
\begin{subremarks}
    \item In particular, the group $\SaQellbarprime$ is a normal subgroup of $\SaQellbar$ and the quotient group $\SaQellbar/\SaQellbarprime$ is isomorphic to $\nprimea \ZnZ$ and hence of order $\da$.
    \item\label{subremark:da.nprimea:indt.a} Let us insist on the fact that $\nprimea$, $\da$, $\SaQellbarprime$ and $\ja$ only depend on $a$ and not on the choice of the representative $(a_1,\dots,a_n) \in \ZnZ$. The group $\SigmaQellbar$ also only depends on $a$, but its construction is not canonical as it depends on an arbitrary choice of numbering.
    \item\label{subremark:ja:depends.on.a} Let us also note that if $k \in \ZnZinv$, then $\da[ka] = \da$, $\nprimea[ka] = \nprimea$, $\SaQellbarprime[ka] = \SaQellbarprime$ and $\SaQellbar[ka] = \SaQellbar$, but $\ja[ka] = k\ja$.
\end{subremarks}
\end{remarks}

\subsection{Character values on a transposition \texorpdfstring{$\tau$}{tau}}\label{subsection:Qellbar[G]:character.values.Sa}

\begin{theorem}\label{result:trace.transposition}
For any transposition $\tau \in \symmetricgroup_n$, we have
\begin{equation}\label{formula:trace.transposition}
\traceh{\tau^*}{\HetladQellbarprim{n-2}{\dworkhypersurfaceFqbar}} = (-1)^n\left(\frac{(1-n)^{n-1}+(n-1)}{n} - \delta_n\right)\mathclose{},
\end{equation}
where, as previously, $\delta_n =0$ if $n$ is odd and $\delta_n =1$ if $n$ is even.
\end{theorem}

\begin{proof}
We may assume that $\tau = (1,2)$. We look for the fixed points of $\tau$, i.e. the set of points $[x_1 \colonsep \dots \colonsep x_n]$ such that $[x_1 \colonsep x_2 \colonsep x_3 \colonsep \dots \colonsep x_n] = [x_2 \colonsep x_1 \colonsep x_3 \colonsep \dots \colonsep x_n]$ and $x_1^n + \dots + x_n^n - n \parameter x_1 \dots x_n = 0$. For such a point, we have $x_1^2 = x_2^2$, so that we are in one of the following two cases.
\begin{subproperty}
    \item We have $x_1 = x_2$ and $2x_2^n + x_3^n + \dots + x_n^n - n \parameter x_2^2 x_3 \dots x_n = 0$. The hypersurface of $\varproj{n-2}$ defined by this equation is smooth because $\parameter^n \neq 1$ and its Euler\namedash Poincar\'e characteristic is $(n-2) + \frac{1}{n}[(1-n)^{n-1}+(n-1)]$ (\thrm{\ref{theorem:hirzebruch}}).
    \item We have $x_1 = -x_2 \neq 0$, in which case $x_3 = \dots = x_n = 0$ and $x_1^n + x_2^n = 0$. This can only happen if $n$ is odd and $[x_1 \colonsep \dots \colonsep x_n] = [1 \colonsep -1 \colonsep 0 \colonsep \dots \colonsep 0]$.
\end{subproperty}

The Euler\namedash Poincar\'e characteristic of the fixed-point subvariety of $\tau$ of $\dworkhypersurfaceFqbar$ is thus
\begin{align*}
\eulerpoincarecaract(\dworkhypersurface^\tau)
& = (n-2) + \frac{(1-n)^{n-1}+(n-1)}{n} + 1 - \delta_n\\
& = (n-1) + \frac{(1-n)^{n-1}+(n-1)}{n} - \delta_n,
\end{align*}
and consequently, as $\tau$ is of order $2$ and $\Fq$ is of characteristic $\neq 2$, \thrm{\ref{result:generic.formula.trace.G}} applies:
\begin{align*}
\traceh{\tau^*}{\HetladQellbarprim{n-2}{\dworkhypersurfaceFqbar}}
& = (-1)^{n-1}\Big((n-1) - \eulerpoincarecaract(\dworkhypersurfaceFqbar^{\tau})\Big)\\
& = (-1)^n\biggl(\frac{(1-n)^{n-1}+(n-1)}{n} - \delta_n\biggl). \qedhere
\end{align*}
\end{proof}

\subsection{Sum of the dimensions of the spaces \texorpdfstring{$\isotypiccomponentQellbar{a}$}{Habar} for \texorpdfstring{$a \in \cargroupAtau$}{a in Âtau}}\label{subsection:Qellbar[G]:value.trace.identity}

\begin{proposition}\label{result:compute.dimensions.transposition}
Let $\tau \in \symmetricgroup_n$ be a transposition. Denote by $\cargroupAtau$\label{definition:cargroupA.fixe.s} the set of elements of $\cargroupA$ fixed by~$\tau$. We have
\[\sum_{a \in \cargroupAtau}{\ma} = (-1)^{n-1}\biggl(\frac{(1-n)^{n-1}+(n-1)}{n} - \delta_n\biggr),\]
where, as previously, $\delta_n = 0$ if $n$ is odd and $\delta_n = 1$ if $n$ is even.
\end{proposition}

\begin{proof}
We may assume that $\tau = (1,2)$. Denote by $B$ the set of elements $(b_1,\dots,b_n) \in (\diffens{\ZnZ}{\set{0}})^n$ such that $b_1 = b_2$ and $b_1 + \dots + b_n = 0$. The map $(b_1,\dots,b_n) \mapsto [b_1,\dots,b_n]$ from $B$ to $\cargroupAtau$ is surjective and each element $a \in \cargroupAtau$ has exactly $\ma$ elements in its preimage. We thus have $\sum_{a \in \cargroupAtau}{\ma} = \card{B}$ and conclude thanks to the following lemma.
\end{proof}

\begin{lemma}
Let $r$ be an integer $\geq 2$. The number of $r$-uples $(b_1,\dots,b_r)$ belonging to $(\diffens{\ZnZ}{\set{0}})^r$ such that $b_1 = b_2$ and $b_1 + \dots + b_r = 0$ is
\[(-1)^{r-1}\biggl(\frac{(1-n)^{r-1}+(n-1)}{n} - \delta_n\biggr).\]
\end{lemma}

\begin{proof}
Denote by $u_r$ the number we want to compute. We have $u_2 = \delta_n$ and $u_r + u_{r+1}$ is the number of $(r+1)$-uples $(b_1,\dots,b_r,b_{r+1}) \in (\diffens{\ZnZ}{\set{0}})^r \times \ZnZ$ such that $b_1 = b_2$ and $b_1 + \dots + b_{r+1} = 0$, that is, $u_r + u_{r+1} = (n-1)^{r-1}$. We deduce the announced result by induction on $r$.
\end{proof}

\subsection{Action of \texorpdfstring{$\SaQellbarprime$}{S'a} on \texorpdfstring{$\isotypiccomponentQellbar{a}$}{Habar}}\label{subsection:Qellbar[G]:action.Saprime}

We start with a general result on automorphisms of finite order with trace equal to the dimension of the space.

\begin{lemma}\label{rappel:trace.elt.ordre.n}
Let $\fieldk$ be a field of characteristic zero, $V$ a vector space of finite dimension over $\fieldk$ and $u$ an automorphism of $V$ of finite order. If $\trace u = \dim V$, then $u = \Id_V$.
\end{lemma}

\begin{proof}
Let $M$ be the matrix of $u$ in a certain basis of $V$ over $\fieldk$. The subfield $\fieldk '$ of $\fieldk$ generated by the coefficients of $M$ embeds itself in $\C$; we can thus restrict ourselves to the case $\fieldk = \C$.

Let $\lambda_1$, \dots, $\lambda_m$ (where $m = \dim V$) be the (complex) eigenvalues of $M$, each repeated with multiplicity. They are all roots of unity. As we have, according to the assumptions of the lemma,
\[\abs{\lambda_1 + \dots + \lambda_m} = \abs{\trace u } = m = \abs{\lambda_1} + \dots + \abs{\lambda_m},\]
the $\lambda_i$'s are positively proportional, hence equal. As their sum is $m$, they are all equal to $1$. The endomorphism $u$ of $V$ is thus unipotent; as it is of finite order, it is equal to $\Id_V$.
\end{proof}

\begin{remark}\label{remarque:cas.egalite.majoration}
Let $\fieldk$ be a field having characteristic zero, and $(V_i)_{i\in I}$ a finite sequence of vector space of finite dimensions over $\fieldk$. For each $i \in I$, let $u_i$ be an automorphism of $V_i$ of finite order. If $\sum_{i\in I}{\trace u_i}$ is equal to $\sum_{i\in I}{\dim V_i}$ (respectively to $-\sum_{i\in I}{\dim V_i}$), then $u_i = \Id_{V_i}$ (respectively $u_i = -\Id_{V_i}$) for all $i \in I$. This results from \lmm{\ref{rappel:trace.elt.ordre.n}} applied to the automorphism $u$ of $V = \bigoplus_{i \in I}{V_i}$ which is equal to $u_i$ (respectively to $-u_i$) over $V_i$ for all $i \in I$.
\end{remark}\bigbreak

Let $\tau \in \symmetricgroup_n$ be a transposition. As $\HetladQellbarprim{n-2}{\dworkhypersurfaceFqbar} = \bigoplus_{a\in \cargroupA}{\isotypiccomponentQellbar{a}}$ and as $\tau^*$ sends $\isotypiccomponentQellbar{a}$ into $\isotypiccomponentQellbar{\taua}$, we have
\[\traceh{\tau^*}{\HetladQellbarprim{n-2}{\dworkhypersurfaceFqbar}} = \sum_{a\in \cargroupAtau}{\traceh{\tau^*}{\isotypiccomponentQellbar{a}}}.\]
By \thrm{\ref{result:trace.transposition}} and \prpstn{\ref{result:compute.dimensions.transposition}}, we also have
\[\traceh{\tau^*}{\HetladQellbarprim{n-2}{\dworkhypersurfaceFqbar}} = -\sum_{a\in \cargroupAtau}{\dim \isotypiccomponentQellbar{a}}.\]
We thus deduce from \rmrk{\ref{remarque:cas.egalite.majoration}} that, for each $a \in \cargroupAtau$, $\tau^*$ acts on $\isotypiccomponentQellbar{a}$ by $-\Id_{\isotypiccomponentQellbar{a}}$.

\begin{theorem}\label{result:action.SaQellbarprime}
Consider $a \in \cargroupA$ and $\sigma \in \SaQellbarprime$. If we denote by $\signature(\sigma)$ the signature of $\sigma$, we have
\[\sigma^*|\isotypiccomponentQellbar{a} = \signature(\sigma)\Id_{\isotypiccomponentQellbar{a}}.\]
\end{theorem}

\begin{proof}
The subgroup $\SaQellbarprime$ of $\symmetricgroup_n$ is generated by the transpositions $\tau$ satisfying $\taua = a$ (see \subsctn{\ref{subsection:Qellbar[G]:structure.SaQellbar}}) and we have just seen that $\tau^* | \isotypiccomponentQellbar{a} = - \Id_{\isotypiccomponentQellbar{a}} = \signature(\tau) \Id_{\isotypiccomponentQellbar{a}}$.
\end{proof}

\subsection{Character values on \texorpdfstring{$\groupA\sigma$}{Asigma} where \texorpdfstring{$\sigma$}{sigma} is a product of \texorpdfstring{$n'$}{n'} disjoint cycles of length \texorpdfstring{$d$}{d}}\label{subsection:Qellbar[G]:character.values.Sigmaa}

Let $n'$ and $d$ be integers $\geq 1$ such that $n'd = n$ and let $\sigma \in \symmetricgroup_n$ be a product of $n'$ disjoint cycles of length $d$. Let $\zeta_1$, \dots, $\zeta_n$ be elements of $\unityroots{n}(\Fq)$ such that $\zeta_1\dots\zeta_n = 1$ and denote by $g$ the element $[\zeta_1,\dots,\zeta_n] \sigma$ of $\groupG = \groupA \rtimes \symmetricgroup_n$. Let $O_1$, \dots, $O_{n'}$\label{definition:orbites} be the $n'$ orbits of $\sigma$ in $\set{1,\dots,n}$ and, for each $\zeta \in \unityroots{n}(\Fq)$, denote by $k(\zeta)$\label{definition:kzeta.nprime} the number of $j \in \set{1,\dots,n'}$ such that $\prod_{i \in O_j}{\zeta_i} = \zeta$. The following theorem generalizes \thrm{\ref{result:trace.z1...zn}} (which is recovered by taking $d=1$ and $n' = n$ i.e. $\sigma = \Id$).

\begin{theorem}\label{result:trace:d-cycles}
Under the preceding assumptions,
\[\traceh{g^*}{\HetladQellbarprim{n-2}{\dworkhypersurfaceFqbar}} = \frac{(-1)^n}{n'} \sum_{\zeta \in \unityroots{n'}(\Fq)}{(1-n)^{k(\zeta)}}.\]
\end{theorem}

\begin{proof}
We may assume that $\sigma$ is the product of $((j-1)d+1,\dots,jd)$ for $1 \leq j \leq n'$ and that $O_j = \set{(j-1)d+1,\dots,jd}$. The fixed points of $g$ in $\dworkhypersurface(\Fqbar)$ are the points $[x_1 \colonsep \dots \colonsep x_n]$ of $\dworkhypersurface(\Fqbar)$ such that
\[[\zeta_{\sigma^{-1}(1)}x_{\sigma^{-1}(1)} \colonsep \dots \colonsep \zeta_{\sigma^{-1}(n)}x_{\sigma^{-1}(n)}] = [x_1 \colonsep \dots \colonsep x_n]\]
i.e.
\[[\zeta_{1}x_{1} \colonsep \dots \colonsep \zeta_{n}x_{n}] = [x_{\sigma(1)} \colonsep \dots \colonsep x_{\sigma(n)}].\]
The subscheme $\dworkhypersurfaceFqbar^g$ of these fixed points is thus the disjoint union, over $\lambda \in \Fqbar^*$, of the closed subschemes $Y_{\lambda}$ of $\dworkhypersurfaceFqbar$ defined by
\[(Y_{\lambda}) \quad \begin{system}
x_1^n + \dots + x_n^n - n \parameter x_1 \dots x_n = 0, \\
x_{\sigma(i)} = \lambda \zeta_i x_i \quad \text{for $1 \leq i \leq n$.}
\end{system}\]
Let $j \in \set{1,\dots,n'}$. If $\prod_{i \in O_j}{\zeta_i} \neq \lambda^{-d}$, the second relation shows that $x_i = 0$ for all $i \in O_j$. If $\prod_{i \in O_j}{\zeta_i} = \lambda^{-d}$, we have $\lambda \in \unityroots{nd}(\Fqbar)$ and the second relation shows that
\[\sum_{i \in O_j}{x_i^n} = x_{jd}^n \biggl(\sum_{i=1}^{d}{(\lambda^n)^i}\biggr) = \begin{cases} d x_{jd}^n & \text{if $\lambda \in \unityroots{n}(\Fq)$,} \\ 0 & \text{if $\lambda \notin \unityroots{n}(\Fq)$.}\end{cases}\]

Consider $\lambda \in \Fqbar^*$ and let $\zeta = \lambda^{-d}$ (as $n = n'd$, we have $\zeta^{n'}=1 \iff \lambda^n = 1$). Denote by $J$ the set of $j \in \set{1,\dots,n'}$ such that $\prod_{i \in O_j}{\zeta_i} = \zeta$ and let $y_j = x_{jd}$ for each $j \in J$. If $\zeta \notin \unityroots{n}(\Fq)$, $J$ is empty and hence $Y_{\lambda}$ is empty. Assume now that $\zeta \in \unityroots{n}(\Fq)$. The number of elements of $J$ is $k(\zeta)$. We consider two cases.
\begin{subproperty}
   \item \textsc{First case:} $\zeta \in \unityroots{n'}(\Fq)$. According to what we have just done, the scheme $Y_{\lambda}$ is isomorphic to the hypersurface of $\varprojFqbar{k(\zeta)-1}$ defined by
\begin{align*}
& d\biggl(\sum_{j\in J}{y_j^n}\biggr) = 0 \quad \text{if $J \neq \set{1,\dots,n'}$,} \\
& d(y_1^n+\dots+y_{n'}^n) - n \parameter ' y_1^d \dots y_{n'}^d = 0 \quad \text{if $J = \set{1,\dots,n'}$,}
\end{align*}
where $\parameter '$ is the product of $\parameter$ by an element of $\unityroots{n}(\Fqbar)$. This hypersurface is smooth (because, in the second case, we have $(\parameter ')^n = \parameter^n \neq 1$ and thus $(\parameter ')^{n'} \neq 1$), hence, by \thrm{\ref{theorem:hirzebruch}} \pg{\pageref{theorem:hirzebruch}}, we have
\[\eulerpoincarecaract(Y_{\lambda}) = k(\zeta) - 1 + \frac{(1-n)^{k(\zeta)}+(n-1)}{n} = k(\zeta) + \frac{(1-n)^{k(\zeta)}-1}{n}.\]
   \item \textsc{Second case:} $\zeta \in \diffens{\unityroots{n}(\Fq)}{\unityroots{n'}(\Fq)}$. This time, the scheme $Y_{\lambda}$ is isomorphic to $\varprojFqbar{k(\zeta)-1}$ if $J \neq \set{1,\dots,n'}$ and to the hypersurface of $\varprojFqbar{n'-1}$ defined by $(y_1\dots y_{n'})^d = 0$ if $J = \set{1,\dots,n'}$. In the first case, we have $\eulerpoincarecaract(Y_{\lambda}) = k(\zeta)$. In the second case, we necessarily have $n' \geq 2$ and the Euler\namedash Poincar\'e characteristic of $Y_{\lambda}$ is equal to that of $Y_{\lambda}^\mathrm{red}$, which is the union in $\varprojFqbar{n'-1}$ of the hyperplanes defined by $y_j = 0$, hence
\begin{align*}
\eulerpoincarecaract(Y_{\lambda})
& = \sum_{\substack{L \subset \set{1,\dots,n'} \\ L \neq \emptyset}}{(-1)^{\card{L}-1}(n' - \card{L})} = \sum_{l=1}^{n'}{(-1)^{l-1} \binomial{n'}{l} (n'-l)} \\ & = n' \sum_{l=1}^{n'-1}{(-1)^{l-1}\binomial{n'-1}{l}} = n'(1-(1+(-1))^{n'-1}) = n' = k(\zeta).
\end{align*}
\end{subproperty}

For each $\zeta \in \unityroots{n}(\Fq)$, there exists exactly $d$ values of $\lambda$ such that $\lambda^{-d} = \zeta$. Thus
\begin{align*}
\eulerpoincarecaract(\dworkhypersurfaceFqbar^g)
& = \sum_{\lambda \in \Fqbar^*}{\eulerpoincarecaract(Y_{\lambda})} = d \sum_{\zeta \in \unityroots{n}(\Fq)}{k(\zeta)} + d \sum_{\zeta \in \unityroots{n'}(\Fq)}{\frac{(1-n)^{k(\zeta)}-1}{n}} \\
& = dn' + \sum_{\zeta \in \unityroots{n'}(\Fq)}{\frac{(1-n)^{k(\zeta)}-1}{n'}} = n-1 + \sum_{\zeta \in \unityroots{n'}(\Fq)}{\frac{(1-n)^{k(\zeta)}}{n'}}.
\end{align*}
The order of $g$ divides $nd$ and hence is prime to $q$; thus, by \thrm{\ref{result:generic.formula.trace.G}},
\begin{align*}
\traceh{g^*}{\HetladQellbarprim{n-2}{\dworkhypersurfaceFqbar}}
& = (-1)^{n-1}\Big((n-1) - \eulerpoincarecaract(\dworkhypersurfaceFqbar^g)\Big) \\
& = \frac{(-1)^n}{n'} \sum_{\zeta \in \unityroots{n'}(\Fq)}{(1-n)^{k(\zeta)}}.\qedhere
\end{align*}
\end{proof}

\subsection{Trace of a product \texorpdfstring{$\sigma$}{sigma} of \texorpdfstring{$n'$}{n'} disjoint cycles of length \texorpdfstring{$d$}{d} acting on \texorpdfstring{$\isotypiccomponentQellbar{a}$}{Habar} when \texorpdfstring{$a \in \cargroupAsigma$}{a in Âsigma}}\label{subsection:Qellbar[G]:trace.Sigmaa}

We keep the notations of \subsctn{\ref{subsection:Qellbar[G]:character.values.Sigmaa}}.

\begin{lemma}\label{resutlat:somme.dim.d-cycles}
If $\sigma \in \symmetricgroup_n$ is a product of $n'$ disjoint cycles of length $d$,
\[\sum_{a \in \cargroupA \textup{ such that } \sigma \in \SaprimeQell}{a(\zeta_1,\dots,\zeta_n)\ma} = \frac{(-1)^{n'}}{n'} \sum_{\zeta \in \unityroots{n'}(\Fq)}{(1-n)^{k(\zeta)}}.\]
\end{lemma}

\begin{proof}
Denote by $B$ the set of $(b_1,\dots,b_n) \in (\ZnZnonzero)^n$ such that $b_1+\dots+b_n=0$ and ${}^{\sigma}(b_1,\dots,b_n) = (b_1,\dots,b_n)$. The image of the map $B \to \cargroupA$, $(b_1,\dots,b_n) \mapsto [b_1,\dots,b_n]$ is the set of $a \in \cargroupA$ such that $\sigma \in \SaQellbarprime$; such an element $a$ has exactly $\ma$ elements in its preimage. The sum we must compute can hence be rewritten as
\[\sum_{(b_1,\dots,b_n) \in B}{\zeta_1^{b_1} \dots \zeta_n^{b_n}}.\]

If $(b_1,\dots,b_n) \in B$, all the $b_i$, for $i$ belonging to an orbit $O_j$ of $\sigma$, are equal to a common $c_j \in \ZnZnonzero$ and we have $d(c_1+\dots c_{n'}) = 0$ in $\ZnZ$ i.e. $c_1+\dots+c_{n'} \in n'\ZnZ$. Our sum can thus be rewritten as
\[\sum_{\substack{c_1,\dots,c_{n'} \in \ZnZnonzero \\ c_1+\dots+c_{n'} \in n'\ZnZ}}{\mu_1^{c_1}\dots\mu_{n'}^{c_{n'}}},\]
where $\mu_j = \prod_{i \in O_j}{\zeta_i}$. We conclude by using the following generalization of \lmm{\ref{lemma:compute.sum.induction.A}} (which is recovered by taking $d=1$ and $n'=n$ i.e. $\sigma = \Id$).
\end{proof}

\begin{lemma}\label{lemme:calcul.somme.rec.G}
Let $r$ be an integer $\geq 1$ and $\mu_1$, \dots, $\mu_r$ elements of $\unityroots{n}(\Fq)$. For each $\zeta \in \unityroots{n}(\Fq)$, we denote by $k(\zeta)$ the number of $j \in \set{1,\dots,r}$ such that $\mu_j = \zeta$. We have
\[\sum_{\substack{c_1,\dots,c_{r} \in \ZnZnonzero \\ c_1+\dots+c_{r} \in n'\ZnZ}}{\mu_1^{c_1}\dots\mu_{r}^{c_{r}}} = \frac{(-1)^r}{n'} \sum_{\zeta \in \unityroots{n'}(\Fq)}{(1-n)^{k(\zeta)}}.\]
\end{lemma}

\begin{proof}
We prove the result by induction on $r$. For $r=1$, we have
\[\sum_{c_1 \in  \diffens{n'\ZnZ}{\set{0}}}{\mu_1^{c_1}} = \begin{cases} d-1 = \frac{-1}{n'} ((1-n)^1+(n'-1)(1-n)^0) & \text{if $\mu_1 \in \unityroots{n'}(\Fq)$,} \\ -1 = \frac{-1}{n'} (n'(1-n)^0) & \text{if $\mu_1 \notin \unityroots{n'}(\Fq)$,}\end{cases}\]
hence the result in that case. Assume now that $r \geq 2$ and that the result is proved for $r-1$. We write
\begin{align*}
& \sum_{\substack{c_1,\dots,c_{r} \in \ZnZnonzero \\ c_1+\dots+c_{r} \in n'\ZnZ}}{\mu_1^{c_1}\dots\mu_{r}^{c_{r}}} + \sum_{\substack{c_1,\dots,c_{r-1} \in \ZnZnonzero \\ c_1+\dots+c_{r-1} \in n'\ZnZ}}{\mu_1^{c_1}\dots\mu_{r-1}^{c_{r-1}}} \\ & \qquad = \sum_{\substack{c_1,\dots,c_{r-1} \in \ZnZnonzero \\ c_r \in \ZnZ \\ c_1+\dots+c_{r} \in n'\ZnZ}}{\mu_1^{c_1}\dots\mu_{r}^{c_{r}}} \\
& \qquad = \sum_{\substack{c_1,\dots,c_{r-1} \in \ZnZnonzero \\ l \in n'\ZnZ}}{\mu_1^{c_1}\dots\mu_{r-1}^{c_{r-1}} \mu_r^{l-c_1-\dots-c_{r-1}}} \\
& \qquad = \sum_{\substack{c_1,\dots,c_{r-1} \in \ZnZnonzero}}{\left(\frac{\mu_1}{\mu_r}\right)^{c_1} \dots \left(\frac{\mu_{r-1}}{\mu_r}\right)^{c_{r-1}} \sum_{l \in n'\ZnZ}{\mu_r^l}}.
\end{align*}
The sum $\sum_{l \in n'\ZnZ}{\mu_r^l}$ is equal to $d$ if $\mu_r \in \unityroots{n'}(\Fq)$ and to $0$ otherwise whereas $\sum_{\substack{c_i \in \ZnZnonzero}}{(\frac{\mu_i}{\mu_r})^{c_i}}$ is equal to $n-1$ if $\mu_i = \mu_r$ and to $-1$ otherwise. The product of all these sums is thus equal to $(-1)^{r-1}d(1-n)^{k(\mu_r)-1}$ if $\mu_r \in \unityroots{n'}(\Fq)$ and to $0$ otherwise.

Taking into account the induction assumption, we obtain
\begin{align*}
&\sum_{\substack{c_1,\dots,c_{r} \in \ZnZnonzero \\ c_1+\dots+c_{r} \in n'\ZnZ}}{\mu_1^{c_1}\dots\mu_{r}^{c_{r}}} \\
& \qquad =  \sum_{\substack{\zeta \in \unityroots{n'}(\Fq)\\ \zeta \neq \mu_r}}{(-1)^r\frac{(1-n)^{k(\zeta)}}{n'}} \\
& \hspace{6em} + \sum_{\substack{\zeta \in \unityroots{n'}(\Fq)\\ \zeta = \mu_r}}{\left((-1)^r\frac{(1-n)^{k(\zeta)-1}}{n'} - d(-1)^r(1-n)^{k(\zeta)-1}\right)} \\
& \qquad = \frac{(-1)^r}{n'} \sum_{\zeta \in \unityroots{n'}(\Fq)}{(1-n)^{k(\zeta)}}.\qedhere
\end{align*}
\end{proof}

\begin{theorem}\label{result:action.d-cycle}
If $\sigma$ is a product of $n'$ disjoint cycles of length $d$ and if $a \in \cargroupAsigma$, then
\[\traceh{\sigma^*}{\isotypiccomponentQellbar{a}} = \begin{cases} (-1)^{n-n'}\ma & \text{if $\sigma \in \SaQellbarprime$,} \\ 0 & \text{if $\sigma \in \diffens{\SaQellbar}{\SaQellbarprime}$.}\end{cases}\]
\end{theorem}

\begin{proof}
As $\HetladQellbarprim{n-2}{\dworkhypersurfaceFqbar} = \bigoplus_{a \in \cargroupA}{\isotypiccomponentQellbar{a}}$ and as $\sigma^*$ sends $\isotypiccomponentQellbar{a}$ into $\isotypiccomponentQellbar{\sigmaa}$, we have, for each $(\zeta_1,\dots,\allowbreak\zeta_n) \in \unityroots{n}(\Fq)^n$ satisfying $\zeta_1\dots\zeta_n=1$,
\[\traceh{([\zeta_1,\dots,\zeta_n]\sigma)^*}{\HetladQellbarprim{n-2}{\dworkhypersurfaceFqbar}} = \sum_{a \in \cargroupAsigma}{a(\zeta_1,\dots,\zeta_n)\traceh{\sigma^*}{\isotypiccomponentQellbar{a}}}.\]
Moreover, by \thrm{\ref{result:trace:d-cycles}} and \lmm{\ref{resutlat:somme.dim.d-cycles}},
\[\sum_{a \in \cargroupA \textup{ such that } \sigma \in \SaQellbarprime}{(-1)^{n-n'}\ma\,a(\zeta_1,\dots,\zeta_n)} = \sum_{a \in \cargroupAsigma}{\traceh{\sigma^*}{\isotypiccomponentQellbar{a}}\,a(\zeta_1,\dots,\zeta_n)}\]
As this is valid for all $(\zeta_1,\dots,\zeta_n) \in \unityroots{n}(\Fq)^n$ satisfying $\zeta_1\dots\zeta_n=1$, we may identify the coefficients, which gives the announced result.
\end{proof}

\subsection{Action of \texorpdfstring{$\SaQellbar$}{Sa} on \texorpdfstring{$\isotypiccomponentQellbar{a}$}{Habar}}\label{subsection:Qellbar[G]:action.Sigmaa}

Let's recapitulate the results of \subsctns{\ref{subsection:Qellbar[G]:character.values.Sa}}\rangedash\ref{subsection:Qellbar[G]:trace.Sigmaa}. We keep the notations of \subsctn{\ref{subsection:Qellbar[G]:structure.SaQellbar}}: $a = [a_1,\dots,a_n]$ is an element of $\cargroupA$, $\nprimea\ZnZ$ is the set of $j \in \ZnZ$ such that $(a_1+j,\allowbreak\dots,\allowbreak a_n+j)$ is a permutation of $(a_1,\dots,a_n)$ and $\da = n/\nprimea$; the fixator $\SaQellbar$ of $a$ in $\symmetricgroup_n$ can be written as
\[\arraycolsep=0pt\begin{array}{lcl}\SaQellbar = \SaQellbarprime \rtimes \SigmaQellbar & \text{\quad where\quad} & \text{$\SaQellbarprime$ is the fixator of $(a_1,\dots,a_n)$ in $\symmetricgroup_n$,} \\[3pt]
& \text{\quad and\quad}  &\text{$\SigmaQellbar = \langle \sigma \rangle$ is a cyclic group of order $\da$,}
\end{array}\]
with $\sigma$ a product of $\nprimea$ disjoint cycles of length $\da$.

The dimension $\ma$ of $\isotypiccomponentQellbar{a}$ is, by \thrm{\ref{result:mult.a}}, equal to $\card{(\diffens{\ZnZ}{\set{a_1,\dots,a_n}})}$. It is a multiple of $\da$ as $\set{a_1,\dots,a_n}$ is stable by translation by elements of $\nprimea\ZnZ$; we can thus write $\ma = \da\mprimea$.\label{definition:mprimea}

\begin{theorem}
The group $\SaQellbar$ acts on $\isotypiccomponentQellbar{a}$ as follows:
\begin{itemize}
    \item an element $s \in \SaQellbarprime$ acts by $\signature(s)\Id_{\isotypiccomponentQellbar{a}}$;
    \item an element $s \in \SigmaQellbar$ acts by $\mprimea$ copies of the regular representation of $\SigmaQellbar$.
\end{itemize}
\end{theorem}

\begin{proof}
The first assertion results from \thrm{\ref{result:action.SaQellbarprime}} and the second from \thrm{\ref{result:action.d-cycle}}: the trace of $\genSigmaQellbar^i$ acting on $\isotypiccomponentQellbar{a}$ is zero if $1 \leq i \leq n-1$ and equal to $\ma = \dim \isotypiccomponentQellbar{a}$ if $i=0$ (note that $(-1)^{n-\nprimea} = 1$ since both $n$ and $\nprimea$ are odd), hence $\SigmaQellbar$ acts as $\mprimea = \ma/\da$ copies of the regular representation.
\end{proof}

This completely determines the structure of the $\Qellbar[\SaQellbar]$-module $\isotypiccomponentQellbar{a}$. From the considerations of \subsctn{\ref{subsection:Qellbar[G]:isotypic.decomposition.Qellbar[G]}}, we deduce the structure of the $\Qellbar[\groupG]$-module $\HetladQellbarprim{n-2}{\dworkhypersurfaceFqbar}$:
\begin{equation}\label{formule:decomposition.Het.QellbarG}
\HetladQellbarprim{n-2}{\dworkhypersurfaceFqbar} \iso \bigoplus_{a \in R}{\Ind_{\groupA \rtimes \SaQellbar}^{\groupG}{(a \tensor \signature \tensor \reg_{\SaQellbar/\SaQellbarprime}^{\ma'}})},
\end{equation}
where\label{definition:reg} $\reg_{\SaQellbar/\SaQellbarprime}$ is the regular representation of $\SaQellbar/\SaQellbarprime$ (let us recall that $R \subset \cargroupA$ is a set of representative elements of $\Sn\backslash\cargroupA$; see \subsctn{\ref{subsection:Qellbar[G]:isotypic.decomposition.Qellbar[G]}}).

\section{Action of \texorpdfstring{$\groupG$}{G} on \texorpdfstring{$\HetladQellprim{n-2}{\dworkhypersurfaceFqbar}$}{Hn-2(X,Ql)prim}}\label{section:Qell[G]}

We begin by giving a canonical construction of cyclotomic fields and characters attached to cyclic groups.

\subsection{The cyclotomic field attached to a cyclic group}\label{subsection:cycl.grp-field}

Let $C$ be a cyclic group of order $m \geq 1$. Denote by $\Q[C]$ the group algebra of $C$ over $\Q$ and by $\mathfrak{m}_C$ the ideal of $\Q[C]$ generated by the sums $\sum_{x \in C'}{[x]}$ for $C'$ a subgroup $\neq \set{1}$ of $C$.

\begin{theorem}
The ideal $\mathfrak{m}_C$ of $\Q[C]$ is maximal and the field $\fieldcyclgroup{C} = \Q[C]/\mathfrak{m}_C$ is isomorphic to the cyclotomic field $\Q(\unityroots{m})$ of \nth{$m$} roots of unity.
\end{theorem}

\begin{proof}
We may assume that $C = \ZmZ$ so that the algebra $\Q[C]$ can be identified with $\Q[X]/(X^m-1)\Q[X]$. We have $X^m-1 = \prod_{d\divides m}{\Phi_d}$, where $\Phi_d$ is the \nth{$d$} cyclotomic polynomial. The polynomials $\Phi_d$ are paiwise prime in $\Q[X]$. From the chinese remainder theorem, we deduce that $\Q[X]/(X^m-1)\Q[X]$ is isomorphic to $\prod_{d\divides m}{\Q[X]/\Phi_d\Q[X]}$. We now proceed to show that $\mathfrak{m}_C$ is the kernel of the projection $\phi \colon \Q[X]/(X^m-1)\Q[X] \to \Q[X]/\Phi_m\Q[X]$. Let $d \neq m$ be an integer dividing $m$ and $C_d = d\ZmZ$ the unique subgroup of $C$ with index $d$; the element $\sum_{x \in C_d}{[x]}$ of $\Q[C]$ has projection $0$ on $\Q[X]/\Phi_m\Q[X]$ and projection $\neq 0$ (equal to $m/d$) on $\Q[X]/\Phi_d\Q[X]$, which shows the result.
\end{proof}

The field $\fieldcyclgroup{C}$\label{definition:corpsgpecycl} is called \emph{the cyclotomic field attached to the cyclic group $C$}. The compound map
\[C \to \Q[C] \to \fieldcyclgroup{C} = \Q[C]/\mathfrak{m}_C\]
is a canonical character $\carcyclgroup{C}$\label{definition:cargpecycl} of $C$ taking values in $\fieldcyclgroup{C}$. It induces an isomorphism between $C$ and the group of \nth{$m$} roots of unity of $\fieldcyclgroup{C}$.

\begin{proposition}
The field $\fieldcyclgroup{C}$ is a simple $\Q[C]$-module with endomorphism ring $\fieldcyclgroup{C}$.
\end{proposition}

Let $C_1$ and $C_2$ be two cyclic groups of same order $m$ and $\phi \colon C_1 \to C_2$ and isomorphism of $C_1$ onto $C_2$. The homomorphism $\Q[C_1] \to \Q[C_2]$ extending $\phi$ factors as an isomorphism $\fieldcyclgroup{\phi} \colon \fieldcyclgroup{C_1} \to \fieldcyclgroup{C_2}$ and we have $\fieldcyclgroup{\phi} \circ \carcyclgroup{C_1} = \carcyclgroup{C_2} \circ \phi$, i.e. the following diagram is commutative
\[\begin{CD}
C_1 @>{\phi}>> C_2\\
@V{\chi_{C_1}}VV @VV{\chi_{C_2}}V\\
\K_{C_1} @>{\K_{\phi}}>> \K_{C_2}
\end{CD}\]

\subsection{The simple \texorpdfstring{$\Q[\groupA]$}{Q[A]}-module attached to an element of \texorpdfstring{$\ZnZinv\backslash\cargroupA$}{(Z/nZ)*\textbackslash Â}}

The group $\ZnZinv$ acts on $\cargroupA$ by $k \times [a_1,\dots,a_n] = [ka_1,\dots,ka_n]$. If $a \in \cargroupA$, we denote by $\classZnZinv{a}$\label{definition:classeZnZinv} the class mod $\ZnZinv$ of $a$. Let us note that the integers $\da$ and $\nprimea$ defined in \subsctn{\ref{subsection:Qellbar[G]:structure.SaQellbar}} only depend on $\classZnZinv{a}$ and not on $a$ (see \rmrk{\ref{remark:da.nprimea:indt.a}}).

Denote by $\na$\label{definition:na} the order of $a$ in the group $\cargroupA$; it only depends on $\classZnZinv{a}$ and not on $a$. If $m$ is an integer, we have $ma = 0$ if and only if all the $ma_i$ are equal, i.e. if and only if $m(a_i-a_{i'}) = 0$ for all $i$ and $i'$ between $1$ and $n$. The subgroup of $\ZnZ$ generated by the elements $a_i-a_{i'}$ only depends on $\classZnZinv{a}$ and not on $a$ or on the choice of $a_1$, \dots, $a_n$; it can be written as $\fa \ZnZ$\label{definition:fa} where $\fa$ divides $n$ and its order is $\na$, hence $n = \na \fa$. The integer $\fa$ only depends on $\classZnZinv{a}$, not on $a$.

Following \subsctn{\ref{subsection:Qellbar[A]:cargroup.A.Qellbar}}, we identify the group $\cargroupA$ to the group of characters of $\groupA$ taking values in $\Fq$, the element $a \in \cargroupA$ corresponding to the character $[\zeta_1,\dots,\zeta_n] \mapsto \zeta_1^{a_1}\dots\zeta_n^{a_n}$. If $\kerclassZnZinv{a}$\label{definition:Na} and $\imgclassZnZinv{a}$\label{definition:Ea} denote the kernel and the image of this character, $\imgclassZnZinv{a} \iso \groupA/\kerclassZnZinv{a}$ is a cyclic subgroup of order $\na$. Let us note that $\imgclassZnZinv{a}$ and $\kerclassZnZinv{a}$ only depend on $\classZnZinv{a}$, not on $a$.

Denote by $\fieldcyclgroupZnZinv{a}$\label{definition:corpsKa} the cyclotomic field attached to the cyclic group $\imgclassZnZinv{a}$ (see \subsctn{\ref{subsection:cycl.grp-field}}) and $\carcyclgroupZnZinv{a}$\label{definition:cara} the compound character
\[\groupA \onto \groupA/\kerclassZnZinv{a} \isoto \imgclassZnZinv{a} \into \fieldcyclgroupZnZinv{a},\]
where the third arrow is the canonical character of $\imgclassZnZinv{a}$ from \subsctn{\ref{subsection:cycl.grp-field}}.

\begin{remarks}\label{remarque:ka=a}\label{remarque:Ka.cara.ind.a}
\begin{subremarks}
    \item Consider $k \in \ZnZinv$. We have $ka = a$ if and only if $k \equiv 1 \mod{\na\Z}$.
    \item The cyclotomic field $\fieldcyclgroupZnZinv{a}$ only depends on $\classZnZinv{a}$ and not on $a$, but $\carcyclgroupZnZinv{ka} = \carcyclgroupZnZinv{a}^k$.
\end{subremarks}
\end{remarks}

\begin{proposition}
The character $\carcyclgroupZnZinv{a}$ defines a structure of simple $\Q[\groupA]$-module on $\fieldcyclgroupZnZinv{a}$ whose endomorphism ring is canonically isomorphic to the field $\fieldcyclgroupZnZinv{a}$.
\end{proposition}

\subsection{The stabilizer \texorpdfstring{$\SaQell$}{Sabar} in \texorpdfstring{$\symmetricgroup_n$}{Sn} of an element \texorpdfstring{$\classZnZinv{a} \in \ZnZinv \backslash \cargroupA$}{abar in (Z/nZ)*\textbackslash Â}}\label{subsection:structure.SaQell}

The group $\Sn$ acts on $\cargroupA$ by ${}^\sigma[a_1,\dots,a_n] = [a_{\sigma^{-1}(1)},\dots,a_{\sigma^{-1}(n)}]$. This action commutes to that of $\ZnZinv$ and factors as an action of $\Sn$ on $\ZnZinv \backslash \cargroupA$. We designate by $\SaQell$\label{definition:SaQell} the fixator of $\classZnZinv{a}$ in $\symmetricgroup_n$.

If $\sigma \in \SaQell$, there exists a unique $k \in \ZnaZinv$ such that $\sigmaa = ka$; we denote it by $\ka(\sigma)$\label{definition:ka}. The map $\ka \colon \SaQell \to \ZnaZinv$ defined in that way is a group homomorphism which is not surjective in general\footnote{\label{footnote:pas.surj}Consider $n=5$ and $a = [0,0,1,1,3]$: we have $\na = 5$, but there is no $\sigma \in \symmetricgroup_5$ such that $\sigmaa = 2a$.}. Its kernel is the group $\SaQellbar$ from \subsctn{\ref{subsection:Qellbar[G]:structure.SaQellbar}}; in particular, $\SaQellbar$ is a normal subgroup of $\SaQell$. Let us note that the map $\ka$ only depends on $\classZnZinv{a}$, not on $a$.

From the definition of $\nprimea$, there is an $i$ such that $a_1 = a_i + \nprimea$, i.e. $\nprimea = a_1-a_i \in \fa\ZnZ$. Thus, there is an integer $\ea$\label{definition:ea} such that $\nprimea = \ea \fa$ and we have $n = \da \ea \fa$ and $\na = \da \ea$. The integer $\ea$ only depends on $\classZnZinv{a}$, not on $a$.

\begin{theorem}
The image of the homomorphism $\ka \colon \SaQell \to \ZnaZinv$ contains the elements of $\ZnaZinv$ which are $\equiv 1 \mod \ea$ and is thus the preimage of a subgroup of $\ZeaZinv$ by the canonical surjection $\ZnaZinv \to \ZeaZinv$.
\end{theorem}

\begin{proof}
Given $k \in \ZnZinv$ such that $k \equiv 1 \mod \ea$, we must find a permutation $\sigma \in \symmetricgroup_n$ such that $\sigmaa = ka$. We only need to show that there exists $j$ such that, for all $b \in \ZnZ$, the sets $I(kb+j)$ and $I(b)$ have the same number of elements. The following lemma shows that we may take $j = -ka_1+a_1$.
\end{proof}

\begin{lemma}
If $k \equiv 1 \mod \ea$, then, for all $b \in \ZnZ$, $I(kb-ka_1+a_1)$ has the same number of elements as $I(b)$.
\end{lemma}

\begin{proof}
Consider $b \in \ZnZ$. Suppose that $b \equiv a_1 \mod{\fa}$, so that $(kb-ka_1+a_1) - b = (k-1)(b-a_1)$ is a multiple of $\ea \fa = \nprimea$ and thus $kb - ka_1 + a_1 \equiv b \mod{\nprimea}$; by \rmrk{\ref{remark:def.naprime}}, this implies that $I(kb - ka_1 + a_1)$ has the same number of elements as $I(b)$.

Suppose now that $b \not\equiv a_1 \mod{\fa}$ (and thus $I(b) = \emptyset$); in that case, $kb-ka_1$ is non zero mod $\fa$ and so, from the definition of $\fa$, $kb - ka_1 + a_1$ is not one of the $a_i$'s, which shows that $I(kb - ka_1 + a_1)$ is empty.
\end{proof}

We now determine the structure of $\SaQell$. Let us recall (see \rmrk{\ref{remark:da.nprimea:indt.a}}) that $\SaQellbarprime$ and $\SaQellbar$ depend only on $\classZnZinv{a}$, not on $a$.

\begin{theorem}
The group $\SaQellbarprime$ is a normal subgroup of $\SaQell$ and the following short exact sequence splits
\[1 \to \SaQellbarprime \to \SaQell \to \SaQell/\SaQellbarprime \to 1.\]
\end{theorem}

\begin{proof}
From the definition of $\fa$, it is possible to choose the representative $(a_1,\dots,a_n)$ of $a$ in $(\ZnZ)^n$ such that each $a_i$ is a multiple of $\fa$; because $\fa \na = n$, the elements $wa_i$ and $w\fa$, where $w \in \ZnaZinv$, are well-defined in $\ZnZ$. If $\sigma \in \SaQell$, there is a unique pair $(\usigma,\vsigma) \in \ZnaZ \times \ZnaZinv$\label{definition:u.v} such that, for all $i$, we have $a_{\sigma(i)} = \vsigma a_i + \usigma \fa$. The uniqueness of $\vsigma$ comes from the fact that, as we have already seen (\rmrk{\ref{remarque:ka=a}}), a $k$ such that $ka = \sigmaa$ is defined mod $\na$ and the uniqueness of $\usigma$ comes from the fact that $\usigma \fa$ is unique mod $n$.

The map $\phi \colon \sigma \mapsto (\usigma,\vsigma)$\label{definition:phi.sigma.u.v} is a group homomorphism from $\SaQell$ to $\ZnaZ \rtimes \ZnaZinv$ (the group law being $(u,v)(u',v') = (u+vu',vv')$); its kernel is $\SaQellbarprime$ which is thus a normal subgroup of $\SaQell$.

For each $b \in \ZnZ$, we choose a numbering $i_1(b)$, \dots, $i_{\card{I(b)}}(b)$ of the elements of $I(b)$. Given $(u,v) \in \phi(\SaQell)$, if $I(b)$ is non-empty, then $b$ is a multiple of $\fa$ (by assumption) and $I(b)$ has the same number of elements than $I(vb+u\fa)$ as $a_{\sigma(i)} = v a_i + u \fa$ for all $\sigma \in \SaQell$ satisfying $\phi(\sigma) = (u,v)$. Thus, there is a permutation $\sigma_{u,v} \in \Sn$ sending $i_l(b)$ on $i_l(vb+u\fa)$ for all $b \in \ZnZ$ and $1 \leq l \leq \card{I(b)}$. From its definition, this permutation belongs to $\SaQell$ and $\phi(\sigma_{u,v}) = (u,v)$. Moreover, the map $(u,v) \mapsto \sigma_{u,v}$ is a group homomorphism since we have
\[v'(vb+u\fa)+u'\fa = (v'v)b + (u'+v'u)\fa.\]
This shows that $(u,v) \mapsto \sigma_{u,v}$ is a splitting map for $\phi$ and thus the short exact sequence $1 \to \SaQellbarprime \to \SaQell \to \SaQell/\SaQellbarprime \to 1$ splits.
\end{proof}

\begin{remarks}\label{remarque:lien.ja.ua}
\begin{subremarks}
    \item Even though $\SaQellbar$ is a normal subgroup of $\SaQell$, the exact short sequence $1 \to \SaQellbar \to \SaQell \to \SaQell/\SaQellbar \to 1$ does not always splits. Indeed, consider the case $n = 24$ and the sequence $(a_1,\dots,a_{24})$ with four times each of the numbers $0$, $2$, $12$, $14$ and two times each of the numbers $1$, $7$, $13$, $19$; we have $\na = 24$, but, even though $5$ is of order $2$ in $\ZnZinv[24]$, the only elements $(u,v)$ of the image of $\phi$ such that $v=5$ are $(2,5)$ and $(14,5)$ which are of order $4$.
    \item \label{subremark:ua.ea} When $\sigma \in \SaQellbar$, we have $\vsigma = 1$ and $\usigma \in \ea\ZnaZ$; indeed, if $\sigma \in \SaQellbar$, then $\vsigma = 1$ and so $a_{\sigma(i)} - a_i = \usigma \fa$; thus, from the definition of $\nprimea$, $\usigma \fa$ is a multiple of $\nprimea = \ea \fa$ and hence $\usigma$ is a multiple of $\ea$.
    \item\label{subremark:link.ja.ua} With the notations of \subsctn{\ref{subsection:Qellbar[G]:structure.SaQellbar}}, we have, for all $s \in\SaQellbar$, $\ja(s) = \fa \us$. More precisely, $\ja \colon \SaQellbar \to \nprimea \ZnZ$ is the compound of the homomorphism $\sigma \mapsto \us$ sending $\SaQellbar$ into $\ea\ZnaZ$ and of the isomorphism of $\ea\ZnaZ$ onto $\nprimea\ZnZ$ deduced from the multiplication by $\fa$.
\end{subremarks}
\end{remarks}

\subsection{Construction of \texorpdfstring{$\Q[\groupG]$}{Q[G]}-modules and study of their extension of scalars to \texorpdfstring{$\Qellbar$}{Qlbar}}\label{subsection:Qell[G].ext.Qellbar}

The aim of this~\subsctn{\ref{subsection:Qell[G].ext.Qellbar}} is to construct $\Q[\groupG]$-modules which, after extension of scalars to $\Qellbar$, will give back the representations considered in \sctn{\ref{section:Qellbar[G]}}.

Before we begin, let us recall that the field $\Ka$ only depends on $\classZnZinv{a}$, not on $a$, but that $\cara[k] = \cara^k$ (see \rmrk{\ref{remarque:Ka.cara.ind.a}}). If $v \in \ZnaZinv$, we denote by $\theta_v$\label{definition:morphisme.theta.v} the automorphism of the field $\Ka$ sending every \nth{$\na$} root of unity onto its \nth{$v$} power.

Consider $a \in \cargroupA$; we choose a representative  $(a_1,\dots,a_n) \in (\ZnZ)^n$ of $a$ such that the $a_i$ are all multiple of $\fa$ and continue to use the notations of~\subsctn{\ref{subsection:structure.SaQell}} concerning the integers $\usigma$ and $\vsigma$.

\begin{proposition}
If $\omega$\label{definition:omega} is a \nth{$\na$} root of unity in $\Ka$, the following map defines a representation of $\groupA \rtimes \SaQell$ into $\Ka$\label{definition:mu.a.omega}
\[\fullapp{\mua}{\groupA \rtimes \SaQell}{\End_\Q(\Ka)}{(\zeta,\sigma)} {\cara(\zeta)\signature(\sigma)\omega^{\usigma}\theta_{\vsigma}}\]
Let $\Mua$\label{definition:Mua} be the $\Q[\groupA \rtimes \SaQell]$-module $\Ka$ thus defined. It has rank $\eulerphi(\na)$ (where $\eulerphi$ is Euler's totient function), and, up to isomorphism, it is independent of the choice of the representative $(a_1,\dots,a_n)$ of $a$ such that each $a_i$ is divisible by $\fa$.
\end{proposition}

\begin{proof}
Let us first check that $\mua$ is a group homomorphism. We have
\begin{align*}
\mua(\zeta,\sigma)\mua(\zeta ',\sigma ')
& = \cara(\zeta)\signature(\sigma)\omega^{\usigma}\theta_{\vsigma}\cara(\zeta ')\signature(\sigma ')\omega^{\usigmaprime}\theta_{\vsigmaprime} \\
& = \cara(\zeta) \cara(\zeta ')^{\vsigma} \signature(\sigma)\signature(\sigma ')\omega^{\usigma+\usigmaprime\vsigma} \theta_{\vsigma \vsigmaprime},
\end{align*}
and
\begin{align*}
\mua((\zeta,\sigma)(\zeta ',\sigma '))
& = \mua(\zeta\prescriptzeta{\sigma} ',\sigma\sigma ') = \cara(\zeta\prescriptzeta{\sigma} ')\signature(\sigma \sigma ')\omega^{\usigma + \vsigma \usigmaprime}\theta_{\vsigma \vsigmaprime} \\
& = \cara(\zeta)\cara(\prescriptzeta{\sigma} ') \signature(\sigma) \signature(\sigma ')\omega^{\usigma + \vsigma \usigmaprime}\theta_{\vsigma \vsigmaprime}.
\end{align*}
To prove these two quantities are equal, we need to show that $\cara(\prescriptzeta{\sigma} ') = \cara(\zeta ')^{\vsigma}$:
\begin{align*}
\cara(\prescriptzeta{\sigma} ') = \car_{{}^{\sigma^{-1}}a}(\zeta ') = \car_{\vsigma a}(\zeta ') = \cara(\zeta ')^{\vsigma}.
\end{align*}

We now proceed to show that $\mua$ does not depends, up to isomorphism, on the choice of the representative $(a_1,\dots,a_n)$ of $a$ such that each $a_i$ is a multiple of $\fa$. If $(a_1',\dots,a_n')$ is another representative, there exists $j$ such that $a_i' = a_i + j\fa$ for all $i$, and so
\[a_{\sigma(i)}' = a_{\sigma(i)} + j\fa = \vsigma a_i + \usigma \fa + j\fa = \vsigma a_i' + (\usigma + j(1-\vsigma))\fa.\]
Thus, $\vsigma ' = \vsigma$ and $\usigma ' = \usigma + j(1-\vsigma)$, hence
\[\mua '(\zeta,\sigma) = \cara(\zeta) \signature(\sigma) \omega^{\usigma + j(1-\vsigma)} \theta_{\vsigma} = \omega^j \mua(\zeta,\sigma) \, \omega^{-j}.\qedhere\]
\end{proof}

We now study the extension of scalars $\Mua \tensor[\Q] \Qellbar$. We use the isomorphism $t$ from \subsctn{\ref{subsection:Qellbar[A]:cargroup.A.Qellbar}} between $\unityroots{n}(\Fq)$ and $\unityroots{n}(\Qellbar)$; there exists a unique embedding $\iotaa$ of $\Ka$ in $\Qellbar$ such that the following diagram is commutative:
\[\arraycolsep=1.4pt\begin{array}{ccccc}
\imgclassZnZinv{a} & \xinto{\phantom{a}} & \unityroots{n}(\Fq) & \xinto{t} & \unityroots{n}(\Qellbar)\\[3pt]
\big\Vert && && \rotatebox[origin=c]{-90}{$\xinto{}$} \\[5pt]
\imgclassZnZinv{a} & \xinto[]{\hphantom{a}} & \Ka &\xinto[\iotaa]{\hphantom{t}}& \,\;\Qellbar\,\;.
\end{array}\]

This embedding only depends on $\classZnZinv{a}$, not on $a$. Moreover, if we identify $a \in \cargroupA$ to a character $\groupA \to \unityroots{n}(\Fq)$, the following diagram is commutative:
\[\arraycolsep=1.4pt\begin{array}{ccccc}
A & \xinto{a} & \unityroots{n}(\Fq) & \xinto{t} & \unityroots{n}(\Qellbar)\\[3pt]
\big\Vert && && \rotatebox[origin=c]{-90}{$\xinto{}$} \\[5pt]
A & \xinto[\cara]{\hphantom{a}} & \Ka &\xinto[\iotaa]{\hphantom{t}}& \,\;\Qellbar\,\;.
\end{array}\]
In the remainder of this~\subsctn{\ref{subsection:Qell[G].ext.Qellbar}}, we identify $\Ka$ to the subfield $\iotaa(\Ka)$ of $\Qellbar$ thanks to $\iotaa$.

With this identification, we have an isomorphism
\[\fulliso{\delta}{\Ka \tensor[\Q] \Qellbar}{\Qellbar^{\ZnaZinv}}{k \tensor \lambda}{(\theta_{v}(k)\lambda)_{v \in \ZnaZinv}}\]
Because
\begin{align*}
k \tensor \lambda & \xmapsto{\mua(\zeta,\sigma) \tensor \Id_{\Qellbar}} \cara(\zeta)\signature(\sigma)\omega^{\usigma}\theta_{\vsigma}(k) \tensor \lambda \\
& \xmapsto{\hspace{2.6em}\delta\hspace{2.6em}} (\cara(\zeta)^v\signature(\sigma)\omega^{v\usigma}\theta_{v\vsigma}(k)\lambda)_{v \in \ZnaZinv},
\end{align*}
the endomorphism of $\Qellbar^{\ZnaZinv}$ deduced from $\mua(\zeta,\sigma) \tensor \Id_{\Qellbar}$ by the isomorphism $\delta$ is given by
\begin{equation}\label{formule:mua.Qellbar}
(x_v)_{v\in\ZnaZinv} \mapsto (\cara[v](\zeta)\signature(\sigma)\omega^{v\usigma}x_{v\vsigma})_{v\in\ZnaZinv}.
\end{equation}

\begin{proposition}\label{result:Mua.Qellbar}
Let $\ua$ be the homomorphism $\sigma \mapsto \usigma$ of $\SaQellbar$ into $\ea\ZnaZ$; it does not depend on the choice of the representative $(a_1,\dots,a_n)$ of $a$ and we have $\uka = k\ua$ for all $k \in \ZnaZinv$ (see \rmrks{\ref{subremark:link.ja.ua}} and \ref{subremark:ja:depends.on.a}). The $\Qellbar[\groupA \rtimes \SaQell]$-module $\Mua \tensor[\Q] \Qellbar$ is isomorphic to
\[\bigoplus_{k \in \ZnaZinv/\img \ka}{\Ind_{\groupA \rtimes \SaQellbar}^{\groupA \rtimes \SaQell}(ka \tensor \signature \tensor \omega^{u_{ka}})}.\]
\end{proposition}

\begin{proof}
\frml{\eqref{formule:mua.Qellbar}} above shows that the isotypic components of the $\Qellbar[\groupA]$-module $\Mua \tensor[\Q] \Qellbar$ are of the form $ka$ for $k \in \ZnaZinv$ (as in \subsctn{\ref{subsection:Qellbar[A]:cargroup.A.Qellbar}}, we identify $a$ to a character taking values in $\Qellbar$); each of these isotypic components is a direct sum of representations of dimension $1$ isomorphic to $ka$.

Let's now determine the action of the group $\SaQellbar$. As $\SaQellbar[ka] = \SaQellbar$ for all $k \in \ZnZinv$, the group $\SaQellbar$ stablizes each one-dimensional piece isomorphic to $ka$ of the $\Qellbar[\groupA]$-module $\Mua \tensor[\Q] \Qellbar$ and, by \frml{\eqref{formule:mua.Qellbar}}, $\SaQellbar$ acts on a piece isomorphic to $ka$ by multiplication by $\signature(\sigma)\omega^{k\usigma} = \signature(\sigma)\omega^{u_{ka}}$.

This shows that the $\Qellbar[\groupA \rtimes \SaQellbar]$-module $\Mua \tensor[\Q] \Qellbar$ is isomorphic to
\begin{equation}\label{eqn:Ma.tenseur.Qlbar}
\bigoplus_{k \in \ZnaZinv}{(ka \tensor \signature \tensor \omega^{u_{ka}})}.
\end{equation}
From \frml{\eqref{formule:mua.Qellbar}} and the fact that $\SaQell/\SaQellbar = \img\ka = \setst{\vsigma}{\sigma \in \SaQell}$, we have the following isomorphism of $\Qellbar[\groupA \rtimes \SaQell]$-modules:
\[\bigoplus_{k \in \img\ka}{(ka \tensor \signature \tensor \omega^{u_{ka}})} \iso \Ind_{\groupA \rtimes \SaQellbar}^{\groupA \rtimes \SaQell}{(a \tensor \signature \tensor \omega^{u_{a}})}.\]
From this, we get the announced result.
\end{proof}

We deduce the following three corollaries.

\begin{corollary}\label{result:Mua.iso.Muaprime}
Up to isomorphism, $\Mua$ only depends on the \nth{$\da$} root of unity $\omega^{\ea}$. More precisely,
\[\Mua \iso \Muaprime \iff a' \in \ZnZinv a \quad \text{and} \quad \omega^{\ea} = \omega '^{\ea}.\]
\end{corollary}

\begin{proof}
As two representations isomorphic after extension of scalars are also isomorphic before (see \cite[\thrm{29.7}, \pg{200}]{Curtis.Reiner}), we only have to show the result for $\Mua \tensor[\Q] \Qellbar$. From \frml{\eqref{eqn:Ma.tenseur.Qlbar}}, we have
\[\Mua \tensor[\Q] \Qellbar|_A \iso \bigoplus_{k \in \ZnaZinv}{ka},\]
which shows that, if $\Mua \tensor[\Q] \Qellbar \iso \Muaprime \tensor[\Q] \Qellbar$, then $a' \in \ZnZinv a$. Let us now assume that $a' \in \ZnZinv a$ so that $\ea = \eaprime$. Recall (see \rmrk{\ref{subremark:ua.ea}} as well as the proof of \prpstn{\ref{result:SaQellbar.semi-direct.prod}}) that $\ua$ is a surjection of $\SaQellbar$ onto $\ea\ZnaZ$ with $\uka = k\ua$. By \frml{\eqref{eqn:Ma.tenseur.Qlbar}}, we have
\[\Mua \tensor[\Q] \Qellbar|_{\SaQellbar} \iso \signature \tensor \bigoplus_{k \in \ZnaZinv}{\omega^{ku_{a}}},\]
hence, if $\Mua \tensor[\Q] \Qellbar \iso \Muaprime \tensor[\Q] \Qellbar$, we have $\setst{\omega^{k\ua}}{k\in\ZnaZinv} = \setst{\omega '^{k\ua}}{k\in\ZnaZinv}$ and so there exists $\kappa \in \ZnaZinv$ such that $\omega^{\ea} = \omega '^{\kappa\ea}$.

Conversely, we assume that $a' \in \ZnZinv a$ and that there exists $\kappa \in \ZnaZinv$ such that $\omega^{\ea} = \omega '^{\kappa\ea}$ and prove that $\Mua \tensor[\Q] \Qellbar \iso \Muaprime \tensor[\Q] \Qellbar$ if and only if $\kappa = 1$. We write $a' = k'a$ so that we have an isomorphism of $\Qellbar[\groupA \rtimes \SaQellbar]$-modules
\begin{align*}
\Muaprime \tensor[\Q] \Qellbar & \iso \bigoplus_{k \in \ZnaZinv}{(kk'a \tensor \signature \tensor \omega '^{u_{kk'a}})} \\
& \qquad\qquad\qquad\qquad = \bigoplus_{k \in \ZnaZinv}{(\kappa ka \tensor \signature \tensor \omega '^{u_{\kappa ka}})} \\
& \qquad\qquad\qquad\qquad = \bigoplus_{k \in \ZnaZinv}{(\kappa ka \tensor \signature \tensor \omega^{u_{ka}})}.
\end{align*}
This shows that $\Muaprime \tensor[\Q] \Qellbar \iso \Mua \tensor[\Q] \Qellbar$ implies $\kappa = 1$. Conversely, if $\kappa = 1$, the isomorphism from \prpstn{\ref{result:Mua.Qellbar}} shows that
\begin{align*}
\Muaprime \tensor[\Q] \Qellbar
& \iso \bigoplus_{k \in \ZnaZinv/\img\ka}{\Ind_{\groupA \rtimes \SaQellbar}^{\groupA \rtimes \SaQell}(kk'a \tensor \signature \tensor \omega^{u_{kk'a}})} \\
& \iso \bigoplus_{k \in \ZnaZinv/\img\ka}{\Ind_{\groupA \rtimes \SaQellbar}^{\groupA \rtimes \SaQell}(ka \tensor \signature \tensor \omega^{u_{ka}})} \\
& \iso \Mua \tensor[\Q] \Qellbar.\qedhere
\end{align*}
\end{proof}

\begin{corollary}\label{result:decomp:Ma.Qellbar}
For each \nth{$\da$} root of unity $\etaa \in \Ka$, we denote by $\omega(\etaa) \in \Ka$ a \nth{$\na$} root of unity satisfying $\omega(\etaa)^{\ea} = \etaa$. We have an isomorphism of $\Qell[\groupA \rtimes \SaQell]$-modules
\[\bigoplus_{\etaa \in \unityroots{\da}(\Ka)}{\Muaomegaeta \tensor[\Q] \Qellbar} \iso \bigoplus_{k\in\ZnaZinv/\img\ka}{\Ind_{\groupA \rtimes \SaQellbar}^{\groupA \rtimes \SaQell}(ka \tensor \signature \tensor \reg_{\SaQellbar/\SaQellbarprime})}.\]
\end{corollary}

\begin{proof}
According to the previous proposition, we only have to check that, for all $k \in \ZnaZinv$,
\[\bigoplus_{\etaa \in \unityroots{\da}(\Ka)}{\omegaeta^{\uka}} = \reg_{\SaQellbar/\SaQellbarprime}.\]
From \rmrk{\ref{subremark:ua.ea}}, we may write $\ua = \ea \uprimea$ where $\uprimea \colon \SaQellbar \to \Z/\da\Z$ is a group homomorphism. We have $\uprimea(\sigma) = 0 \iff \ua(\sigma) = 0 \iff \sigma \in \SaQellbarprime$ as $\ja = -\fa\ua$ (\rmrk{\ref{subremark:link.ja.ua}}). Consequently, if $\sigma \in \SaQellbar$,
\begin{align*}
\sum_{\etaa \in \unityroots{\da}(\Ka)}{\omegaeta^{\uka(\sigma)}}
& = \sum_{\etaa \in \unityroots{\da}(\Ka)}{\omegaeta^{k\ua(\sigma)}} = \sum_{\etaa \in \unityroots{\da}(\Ka)}{\eta^{k\uprimea(\sigma)}} \\
& = \begin{cases} \da & \text{if $\sigma \in \SaQellbarprime$,} \\ 0 & \text{otherwise,}\end{cases}
\end{align*}
which proves the announced result.
\end{proof}

\begin{corollary}\label{result:decomp.cohom.Qell}
We keep the notations of the previous corollary. We have an isomorphism of $\Qell[\groupG]$-modules
\[\HetladQellprim{n-2}{\dworkhypersurfaceFqbar} \iso \bigoplus_{a \in \ZnZinv \times \Sn \backslash \cargroupA}{\mprimea \Ind_{\groupA \rtimes \SaQell}^{\groupG}\biggl(\mathinner{}\bigoplus_{\etaa \in \unityroots{\da}(\Ka)}{\Muaomegaeta}\biggr)} \tensor[\Q] \Qell.\]
\end{corollary}

\begin{proof}
As a consequence of the previous corollary and of the results of \subsctn{\ref{subsection:Qellbar[G]:action.Sigmaa}}, we have
\[\HetladQellbarprim{n-2}{\dworkhypersurfaceFqbar} \iso \bigoplus_{a \in \ZnZinv \times \Sn \backslash \cargroupA}{\mprimea \Ind_{\groupA \rtimes \SaQell}^{\groupG}\biggl(\mathinner{}\bigoplus_{\etaa \in \unityroots{\da}(\Ka)}{\Muaomegaeta}\biggr)} \tensor[\Q] \Qellbar.\]
We deduce the announced result over $\Qell$ thanks to the same argument as in \crllr{\ref{result:Mua.iso.Muaprime}}: two representations isomorphic after extension of scalars are also isomorphic before.
\end{proof}

\subsection{Endomorphism rings of the representations}%
\label{subsection:computation.end.ring}

Denote by $\Waomega$\label{definition:Wa} the $\Q[\groupG]$-module $\Ind_{\groupA \rtimes \SaQell}^{\groupG}{\Mua}$; the aim of this~\subsctn{\ref{subsection:computation.end.ring}} is to show that it is a simple module and identify its endomorphism ring.

\begin{theorem}
The $\Q[\groupG]$-module $\Waomega$ is simple. Moreover, if we identify the group $\Gal(\Ka/\Q)$ with $\ZnaZinv$, the endomorphism ring of $\Waomega$ identifies with the unique subfield $\Da$\label{definition:Da} of $\Ka$ such that $\Gal(\Ka/\Da) = \img\ka$. That is to say, $\Da$ is the subfield of $\Ka$ consisting of the elements fixed by all the $\theta_{\vsigma}$ for $\sigma \in \SaQell$. In particular, $\Da$ is commutative.
\end{theorem}

\begin{proof}
Since a $\Q[\groupG]$-module is simple if and only if its endomorphism ring is a division ring, we only need to show the second assertion.

We have $\Waomega = \Ind_{\groupA \rtimes \SaQell}^{\groupG}{\Mua}$ where $\Mua$ is just $\Ka$ with the structure of $\Q[\groupA \rtimes \SaQell]$-module given by the representation $\mua$. We may write $\Waomega = \bigoplus_{s \in \Sn/\SaQell}{s\Mua}$. From the definition of $\SaQell$, each $s\Mua$ is stable by $\groupA$ and the $\Q[\groupA]$-modules $s\Mua$ are disjoint. Consequently, the endomorphism ring of $\Waomega$ stabilizes $\Mua$ and $u \mapsto u|_{\Mua}$ defines an isomorphism between the endomorphism ring of $\Waomega$ and the endomorphism ring of the $\Q[\groupA \rtimes \SaQell]$-module $\Mua$.

We now need to show that the endomorphism ring of the $\Q[\groupA \rtimes \SaQell]$-module $\Mua$ is the subfield of $\Ka$ fixed by all the $\theta_{\vsigma}$ for $\sigma \in \SaQell$. The endomorphism ring of the $\Q[\groupA]$-module $\Mua$ is canonically isomorphic to $\Ka$ via $x \mapsto (\lambda \mapsto x\lambda)$ since the $\Q[\groupA]$-module $\Mua$ is $\Ka$. We deduce that the endomorphism ring of the $\Q[\groupA \rtimes \SaQell]$-module $\Mua$ is the subfield of $\Ka$ consisting of the elements $x$ such that $\lambda \mapsto x \lambda$ commutes with each $\mua(\zeta,\sigma)$ i.e. with each $\theta_{\vsigma}$. Because $\lambda \mapsto x\lambda$ commutes with $\theta_{\vsigma}$ if and only if $\theta_{\vsigma}(x) = x$, the ring $\Da = \End_{\Q[\groupG]}(\Waomega,\Waomega)$ is the subfield of $\Ka$ fixed by each $\theta_{\vsigma}$ for $\sigma \in \SaQell$.
\end{proof}

\begin{remarks}\label{remarque:Da.ind.ell}
\begin{subremarks}
    \item The field $\Da$ is independent of the choice of $\omega$.
    \item The field $\Da$ has dimension $\frac{\eulerphi(\na)}{\card{\img\ka}}$ over $\Q$. When $\ZnaZinv$ is cyclic (e.g. when $n$ is prime and $\na = n$), this dimension characterizes $\Da$.
    \item As $\ZeaZinv \subset \img\ka$, we have $\Da \subset \Kprimea$ where $\Kprimea$ is the subfield of $\Ka$ generated by the \nth{$\ea$} roots of unity. In general, $\Da \neq \Kprimea$ as we may see by taking $n = 5$ and $a = [0,0,1,1,3]$: we have $\na = \ea = 5$ and so $\Ka = \Kprimea = \Q(\unityroots{5})$ whereas $\Da = \Q(\sqrt{5})$ (this is the same example as in the footnote to \pg{\pageref{footnote:pas.surj}}).
\end{subremarks}
\end{remarks}

\begin{examples}\label{exemple:calcul.Da}
\begin{subexamples}
    \item\label{sous-exemple:calcul.Da:a=0} When $a = [0,\dots,0]$, we have $\Da = \Ka = \Q$.
    \item\label{sous-exemple:calcul.Da:n=5} When $n = 5$ and $\classZnZinv{a}$ is the class of $[0,0,0,1,4]$ or $[0,0,1,1,3]$, we have $\Da = \Q(\sqrt{5})$.
    \item\label{sous-exemple:calcul.Da:n=7} When $n=7$, we have the following possibilities concerning $\Da$.
\begin{center}\begin{tabular}{cc}
class of $\classZnZinv{a}$ & $\Da$ \\
\noalign{\vspace{1pt}}
\hline
\noalign{\vspace{2pt}}
$[0,0,0,0,0,0,0]$, $[0,1,2,3,4,5,6]$ & $\Q$ \\
\noalign{\vspace{2pt}}
\hline
\noalign{\vspace{2pt}}
$[0,0,0,0,1,2,4]$, $[0,0,1,1,3,3,6]$ & $\Q(\sqrt{-7})$ \\
\noalign{\vspace{2pt}}
\hline
\noalign{\vspace{2pt}}
$[0,0,0,0,0,1,6]$, $[0,0,0,1,1,1,4]$ \\
$[0,0,0,1,1,6,6]$, $[0,0,0,1,2,5,6]$ & $\Q(\unityroots{7})^+$ \\
$[0,0,1,1,3,4,5]$, $[0,0,1,1,2,4,6]$ \\
\noalign{\vspace{2pt}}
\hline
\noalign{\vspace{2pt}}
$[0,0,0,0,1,1,5]$, $[0,0,0,1,1,2,3]$ & $\Q(\unityroots{7})$ \\
\end{tabular}\end{center}
\end{subexamples}
\end{examples}

\begin{theorem}
We have
\[\Waomega \iso \Waomegaprime \iff a \in (\ZnZinv \times \Sn) a' \hskip0.5em \text{and} \hskip0.7em \omega^{\ea} = \omega '^{\ea}.\]
\end{theorem}

\begin{proof}
As two representations isomorphic after extension of scalars are also isomorphic before (see \cite[\thrm{29.7}, \pg{200}]{Curtis.Reiner}), we only need to show the result for $\Waomega \tensor[\Q] \Qellbar$. Following \prpstn{\ref{result:Mua.Qellbar}}, we have
\[\Waomega \tensor[\Q] \Qellbar = \bigoplus_{s\in\Sn/\SaQell}{s\Mua\tensor\Qellbar} \iso \bigoplus_{s\in\Sn/\SaQell}{s\mathinner{}\mathclose{}\biggl(\mathinner{}\bigoplus_{k \in \ZnaZinv}{(ka \tensor \signature \tensor \omega^{u_{ka}})}\biggr)}.\]

If $a$ and $a'$ are the same mod the action of $\ZnZinv \times \Sn$, this formula shows that $\Waomega \tensor \Qellbar$ and $\Waomegaprime \tensor \Qellbar$ are not isomorphic.

If $a \in (\ZnZinv \times \Sn) a'$, as the group $\groupA \rtimes \SaQell$ stabilizes each copy of $s\Mua$ an thus stabilizes $\Mua$, we deduce, thanks to \crllr{\ref{result:Mua.iso.Muaprime}}, that if $\omega^{\ea} \neq \omega '^{\ea}$, then $\Waomega \tensor \Qellbar$ and $\Waomegaprime \tensor \Qellbar$ are not isomorphic.

Finally, if $a \in (\ZnZinv \times \Sn) a'$ and $\omega^{\ea} = \omega '^{\ea}$, then the previous formula shows that $\Waomega \tensor \Qellbar \iso \Waomegaprime \tensor \Qellbar$.
\end{proof}

\section{Consequence for the factorization of the zeta function}\label{section:fact.zeta}

The aim of this \sctn{\ref{section:fact.zeta}} is to show that $\HetladQellprim{n-2}{\dworkhypersurfaceFqbar}$ is a direct sum of subspaces stable by the Frobenius and to deduce a factorization of the zeta function of $\dworkhypersurface$. The idea of using this method comes from \cite[\sctn{6.2}]{HKS}.

The subspaces we consider are the isotypic components of the $\Q[\groupG]$-module $\HetladQellprim{n-2}{\dworkhypersurfaceFqbar}$; after describing them in \subsctn{\ref{subsection:decomp.isotypique}}, we study in \subsctn{\ref{subsection:action.frobenius.composant.iso}} how the Frobenius acts on them and deduce that the characteristic polynomial of the restriction of the Frobenius is an integer power $\Qa^{\gammaa/\da}$ of a polynomial $\Qa$ which has integer coefficients independent of $\ell$ (see \subsctn{\ref{subsection:independance.ell}}). Finally, in \subsctn{\ref{subsection:decomp.fact.zeta.ext.finies}}, we deduce that the part of the zeta function of $\dworkhypersurface$ corresponding to $\HetladQellprim{n-2}{\dworkhypersurfaceFqbar}$ is the product over $a \in \cargroupA$ and $\etaa \in \unityroots{\da}(\Ka)$ of the polynomials $\Qaomegaeta^{\gammaa/\da}$ (see \crllr{\ref{result:decomp:Ma.Qellbar}} for the definition of $\omegaeta$) and we show that each $\Qaomegaeta$ factors over the field $\Da$ considered in \subsctn{\ref{subsection:computation.end.ring}}. We end by explicitly treating the cases $n=3$, $4$, $5$, and $7$ in \subsctn{\ref{subsection:exemples}}.

\subsection{Isotypic decomposition of the \texorpdfstring{$\Qell[\groupG]$}{Ql[G]}-module \texorpdfstring{$\HetladQellprim{n-2}{\dworkhypersurfaceFqbar}$}{Hn-2(X,Ql)prim}}\label{subsection:decomp.isotypique}

The aim of this~\subsctn{\ref{subsection:decomp.isotypique}} is to express, in terms of the representations $\Waomega$ considered above, the isotypic components of the $\Q[\groupG]$-module $\HetladQellprim{n-2}{\dworkhypersurfaceFqbar}$. We keep the notations of \subsctn{\ref{subsection:computation.end.ring}}.

\begin{proposition}
Let $\omega$ be a \nth{$\na$} root of unity. The $\Da \tensor[\Q] \Qell$-module $\Vaomega = \Hom_{\Q[\groupG]}(\Waomega,\HetladQellprim{n-2}{\dworkhypersurfaceFqbar})$\label{definition:Va} is free of rank $\mprimea$.
\end{proposition}

\begin{proof}
By \crllr{\ref{result:decomp.cohom.Qell}}, we have
\[\HetladQellprim{n-2}{\dworkhypersurfaceFqbar} \iso \bigoplus_{a \in \ZnZinv \times \Sn \backslash \cargroupA}{\biggl(\mathinner{}\bigoplus_{\etaa \in \unityroots{\da}(\Ka)}{\Waomegaeta^{\mprimea}} \tensor[\Q] \Qell\biggr)}.\]
We deduce the following isomorphisms of $\Da \tensor[\Q] \Qell$-modules:
\begin{align*}
\Vaomega
& = \Hom_{\Q[\groupG]}(\Waomega,\HetladQellprim{n-2}{\dworkhypersurfaceFqbar}) \\
& \iso \bigoplus_{a' \in \ZnZinv\times\Sn \backslash \cargroupA} {\biggl(\mathinner{}\bigoplus_{\etaa ' \in \unityroots{\da}(\Ka)} {\Hom_{\Q[\groupG]}(\Waomega,\Waprimeomegaetaprime^{\mprimeaprime} \tensor[\Q] \Qell)}\biggr)} \\
& \iso \Hom_{\Q[\groupG]}(\Waomega,\Waomega^{\mprimea} \tensor[\Q] \Qell) \\
& \iso (\End_{\Q[\groupG]}(\Waomega) \tensor[\Q] \Qell)^{\mprimea} \\
& \iso (\Da \tensor[\Q] \Qell)^{\mprimea}.
\end{align*}
This shows that $\Vaomega$ is a free $\Da \tensor[\Q] \Qell$-module of rank $\mprimea$.
\end{proof}

\begin{corollary}\label{result:iso.components.Q}
The map $w \tensor v \mapsto v(w)$ of $\Waomega \tensor[\Da] \Vaomega$ into $\HetladQellprim{n-2}{\dworkhypersurfaceFqbar}$ is $\Qell[\groupG]$-linear and injective; its image is the $\Waomega$-isotypic component $\isotypiccomponentQell{a}$\label{definition:Ha} of the $\Q[\groupG]$-module $\HetladQellprim{n-2}{\dworkhypersurfaceFqbar}$.
\end{corollary}

\begin{proof}
We refer the reader to \cite[\subsctn{3.4}, \prpstn{9}, \pg{33}]{Bourbaki.algebre.viii} and \cite[\subsctn{1.5}, \thrm{1.b}, \pg{15}]{Bourbaki.algebre.viii}.
\end{proof}

\begin{remark}
The link between the $\isotypiccomponentQellbar{\alpha}$ from \subsctn{\ref{subsection:Qellbar[G]:isotypic.decomposition.Qellbar[G]}} and the isotypic components $\isotypiccomponentQell{a}$ from the previous corollary is given by
\[\bigoplus_{\etaa \in \unityroots{\da}(\Ka)}{\isotypiccomponentQell[\omegaeta]{a}} \tensor[\Qell] \Qellbar \iso \bigoplus_{\alpha \in \ZnaZinv/\img\ka}{\Ind_{\groupA \rtimes \SaQellbar}^{\groupG}{\isotypiccomponentQellbar{\alpha}}}.\]
\end{remark}

\subsection{Action of the Frobenius on each isotypic component}\label{subsection:action.frobenius.composant.iso}

\begin{lemma}
The Frobenius stablizes the $\Qell[\groupG]$-modules $\Waomega \tensor[\Da] \Vaomega$.
\end{lemma}

\begin{proof}
As all the elements of $\groupG$ are automorphisms of $\dworkhypersurface$ defined over $\Fq$, the Frobenius endomorphism on $\HetladQell{n-2}{\dworkhypersurface}$ commutes with the action of $\groupG$; it thus stabilizes each isotypic components of the $\Q[\groupG]$-module $\HetladQellprim{n-2}{\dworkhypersurfaceFqbar}$, namely, each of the $\Waomega \tensor[\Da] \Vaomega$ (\crllr{\ref{result:iso.components.Q}}).
\end{proof}

\begin{proposition}
The Frobenius acts on $\Waomega \tensor[\Da] \Vaomega$ by $\Id \tensor \va$ where $\va$\label{definition:va} is the endomorphism $v \mapsto \Frobcohom \circ v$ of the $\Da \tensor[\Q] \Qell$-module $\Vaomega$.
\end{proposition}

\begin{proof}
The action of the Frobenius on $\Waomega \tensor[\Da] \Vaomega$ is given by
\begin{align*}
\Frobcohom(w \tensor v)
& = \Frobcohom(v(w)) = (\Frobcohom \circ v)(w) = \va(v)(w) = w \tensor \va(v) \\
& = (\Id \tensor \va)(w \tensor v).
\end{align*}
The structure of $\Da \tensor[\Q] \Qell$-module of $\Vaomega = \Hom_{\Q[\groupG]}(\Waomega,\HetladQellprim{n-2}{\dworkhypersurfaceFqbar})$ is given by $(d \tensor \lambda) v = \lambda (v \circ d)$. We have
\[\Frobcohom \circ (\lambda (v \circ d)) = \lambda (\Frobcohom \circ v) \circ d,\]
and hence the map $\va$ is an endomorphism of the $\Da \tensor[\Q] \Qell$-module $\Vaomega$.
\end{proof}

We deduce the following result, which describes the reciprocal polynomial of the characteristic polynomial of the Frobenius on each isotypic component.

\begin{proposition}\label{result:action.frob.WaxVa}
Let $\omega$ be a \nth{$\na$} root of unity, and set\label{definition:Pa}\label{definition:Qa}
\begin{align*}
& \Pa(t) = \deth{1-t\va}{\Vaomega/\Da\tensor[\Q]\Qell} \in \Da\tensor[\Q]\Qell[t]; \\
& \Qa(t) = \NormFext{\Da\tensor \Qell[t]}{\Qell[t]}(\Pa(t)) \in \Qell[t].
\end{align*}
We have $\deg \Pa = \mprimea$ and $\deg \Qa = \frac{\eulerphi(\na)}{\card{\img\ka}}\mprimea$. The reciprocal polynomial of the characteristic polynomial of the Frobenius over $\Waomega \tensor[\Da] \Vaomega$ is given by
\[\deth{1-t\Frobcohom}{\Waomega \tensor[\Da] \Vaomega} = \Qa(t)^{\gammaa/\da},\]
where $\gammaa$ is the number of permutations of $(a_1,\dots,a_n)$ and $\da$ is the integer defined in \subsctn{\ref{definition:da}}.
\end{proposition}

\begin{proof}
As $\Frobcohom$ acts on $\Waomega \tensor[\Da] \Vaomega$ by $\Id \tensor \va$, we have \cite[\subsctn{8.6}, \exmpl{3}, \pg{101}]{Bourbaki.algebre.iii}
\begin{align*}
& \deth{1-t\Frobcohom}{\Waomega \tensor[\Da] \Vaomega/\Qell} \\
& \hspace{8em} = \deth{1-t\va}{\Vaomega/\Qell}^{\dim_{\Da}\Waomega} \\
& \hspace{8em} = \deth{1-t\va}{\Vaomega/\Qell}^{(\dim_{\Q}\Waomega)/[\Da:\Q]},
\end{align*}
with \cite[\subsctn{9.4}, \prpstn{6}, \pg{112}]{Bourbaki.algebre.iii}
\[\deth{1-t\va}{\Vaomega/\Qell} = \NormFext{\Da\tensor[\Q] \Qell[t]}{\Qell[t]}(\deth{1-t\va}{\Vaomega/\Da\tensor[\Q]\Qell}),\]
which shows the announced formula given the following remarks:
\begin{subproperty}
   \item the degree of the polynomial $\Pa(t)$ is $\mprimea = \dim_{\Da\tensor[\Q]\Qell}{\Vaomega}$;
   \item the degree of the polynomial $\Qa(t)$ is $[\Da:\Q] \cdot \deg \Pa = \frac{\eulerphi(\na)}{\card{\img \ka}}\mprimea$;
   \item the dimension of $\Waomega$ over $\Q$ is $\eulerphi(\na) [\Sn\colonsep\SaQellbar] = \frac{\eulerphi(\na)}{\card{\img\ka}}\frac{\gammaa}{\da} = \frac{\gammaa}{\da}{[\Da\colonsep\Q]}$, and thus $\frac{\dim_{\Q}\Waomega}{[\Da\colonsep\Q]} = \frac{\gammaa}{\da}$.\qedhere
\end{subproperty}
\end{proof}

\subsection{Rationality and independence of \texorpdfstring{$\ell$}{l} of the characteristic polynomials}\label{subsection:independance.ell}

The aim of this~\subsctn{\ref{subsection:independance.ell}} is to show that the polynomials $\Qa$ defined in \prpstn{\ref{result:action.frob.WaxVa}} have rational coefficients an are independent of $\ell$. We start with the following lemma, which we will use a couple of times in what follows.

\begin{lemma}\label{lemme:det.ind.ell=tr.puiss.r.ind.ell}
Let $E$ be a finite dimensional vector space over $\Qell$ and $u$ an endomorphism of $E$. The polynomial $\det(1-tu)$ is an element of $\Q[t]$ independent of $\ell$ if and only if for all $r \geq 1$ the number $\trace(u^r)$ belongs to $\Q$ and is independent of $\ell$.
\end{lemma}

\begin{proof}
This is a straightforward consequence both of Viete's formulas (relating roots and coefficients of a polynomial) and of Newton's formulas.
\end{proof}

The following lemma allows us to relate the independence of $\ell$ of $\Qa$ to that of $\Qa(t)^{\gammaa/\da}$.

\begin{lemma}\label{result:ind.ell.puiss=>ind.ell}
Let $P \in 1+t\Q[t]$ be a non-constant polynomial and $\gamma \in \N^*$. If, for each $\ell$, there is a $Q_\ell \in 1+t\Qell[t]$ such that $Q_\ell^{\gamma} = P$, then $Q_\ell$ belongs to $1+t\Q[t]$ and is independent of $\ell$.
\end{lemma}

\begin{proof}
Denote by $\sqrt[\gamma]{P}$ the unique element of $1 + t\Q[[t]]$ such that $(\sqrt[\gamma]{P})^\gamma = P$. We have $Q_\ell^\gamma = (\sqrt[\gamma]{P})^\gamma = P$ with $Q_\ell \in 1 + t\Qell[[t]]$, which shows, as $\sqrt[\gamma]{P}$ is unique in $1 + t\Qell[[t]]$, that $Q_\ell = \sqrt[\gamma]{P}$. Consequently, $Q_\ell$ belongs to $1 + t\Q[t]$ and is independent of $\ell$.
\end{proof}

We now deal with the independence of $\ell$ of $\Qa(t)^{\gammaa/\da}$ thanks to an argument of projector.

\begin{proposition}\label{result:Qagamma.ind.ell}
For each $a \in \cargroupA$, the polynomial $\Qa(t)^{\gammaa/\da}$ has rational coefficients and is independent of $\ell$.
\end{proposition}

\begin{proof}
Denote by $\carbis_a \colon g \in G \mapsto \traceh{g^*}{\Waomega/\Q}$ the character of the simple $\Q[\groupG]$-module $\Waomega$. There is a projection $\pi_a$ of $\HetladQellprim{n-2}{\dworkhypersurfaceFqbar}$ onto $\Waomega \tensor[\Da] \Vaomega$ of the form
\[\pi_a = \frac{\lambda}{\card{\groupG}} \sum_{g \in \groupG}{\carbis_a(g^{-1}) g^*}\text{,} \qquad \text{avec $\lambda \in \Q$,}\]
where $\lambda$ is computed by taking the trace of both members of the equality
\[\dim_{\Q}\Waomega = \frac{\lambda}{\card{\groupG}} \sum_{g \in \groupG}{\carbis_a(g^{-1}) \carbis_a(g)} = \lambda [\Da:\Q].\]
(Indeed, over $\Qellbar$, $\carbis_a$ is the direct sum of $[\Da:\Q]$ irreducible characters as we have seen in~\sctn{\ref{section:Qell[G]}}.) We thus have $\lambda = \dim_{\Da}\Waomega$.

Because the image of the projection $\pi_a$ is $\Waomega \tensor[\Da] \Vaomega$, we have
\[\Qa(t)^{\gammaa/\da} = \deth{1-t (\pi_a \circ \Frobcohom)}{\HetladQellprim{n-2}{\dworkhypersurfaceFqbar}}\text{.}\]
Using \lmm{\ref{lemme:det.ind.ell=tr.puiss.r.ind.ell}}, we only have to show that the powers of $\pi_a \circ \Frobcohom$ have a trace belonging to $\Q$ and independent of $\ell$. This results from the fact that these powers can be written as linear combinations with coefficients in $\Q$ of quantities of the type $f^*$ where $f$ is an endomorphism of the variety $\dworkhypersurface$ which extends to $\varproj{n-1}$ and from the following lemma, which is an adaptation of \cite[\thrm{2.2}, \pg{76}]{Katz.Messing} to the case of traces over the primitive part of the cohomology of an irreducible hypersurface (since $n \geq 3$, $\dworkhypersurface$ is irreducible).
\end{proof}

\begin{lemma}\label{result:ind.ell:Hiprime}
Let~$X$ be a non-singular, irreducible hypersurface of $\varproj{n-1}$. If $f \colon X \to X$ is an endomorphism of~$X$ which extends into an endomorphism of $\varproj{n-1}$, then $\traceh{f^*}{\HetladQellprim{n-2}{\Xbar}}$ is an integer which is independent of~$\ell$.
\end{lemma}

\begin{proof}
We have $\HetladQell{n-2}{\Xbar} \iso \HetladQellprim{n-2}{\Xbar} \oplus \HetladQellinprim{n-2}{\Xbar}$ with $\traceh{f^*}{\HetladQell{n-2}{\Xbar}}$ and $\traceh{f^*}{\HetladQellinprim{n-2}{\Xbar}} = \traceh{f^*}{\HetladQell{n-2}{\varprojFqbar{n-1}}}$ two integers independent of $\ell$ by \cite[\thrm{2.2}, \pg{76}]{Katz.Messing}\footnote{On this subject, see also \cite[\pg{119}]{Deligne.Lusztig} and \cite[\subsctn{3.5}, \pgs{112}{113}]{Illusie.miscellaneous}.}.
\end{proof}

Combining \lmm{\ref{result:ind.ell.puiss=>ind.ell}} and \prpstn{\ref{result:Qagamma.ind.ell}}, we deduce the announced result.

\begin{theorem}
The polynomials $\Qa(t)$ have rational coefficients and are independent of~$\ell$.
\end{theorem}

In \subsctn{\ref{subsection:consequences.fact.zeta}}, we will see a stronger result, namely that the polynomials $\Pa$ are independent of $\ell$.

\subsection{Factorization of the zeta function}\label{subsection:consequences.fact.zeta}\label{subsection:decomp.fact.zeta.ext.finies}

From the preceding results, we can deduce a factorization over $\Q$ of the zeta function as well as the existence of a decomposition of some of the factors over finite extensions of $\Q$.

\begin{theorem}
The zeta function of the hypersurface $\dworkhypersurface$ of $\varprojFq{n-1}$ defined by $x_1^n + \dots + x_n^n - n \parameter x_1 \dots x_n = 0$ (with $\parameter \in \Fqnonzero$ satisfying $\psi^n \neq 1$) factors over $\Q$ as
\[\funzeta{\dworkhypersurface/\Fq}{t} = \frac{\displaystyle\Bigl(\prod\nolimits_{a \in \ZnZinv \times \Sn \backslash \cargroupA,\text{ }\eta \in \unityroots{\da}(\Ka)}{\Qaomegaeta(t)^{\gammaa/\da}}\Bigr)^{(-1)^{n-1}}} {(1-t)(1-qt)\dots(1-q^{n-2}t)}.\]
(The notations are those of \crllr{\ref{result:decomp:Ma.Qellbar}} and \prpstn{\ref{result:action.frob.WaxVa}}.)
\end{theorem}

\begin{proof}
The previous formula is just a reformulation of the results from \subsctns{\ref{subsection:decomp.isotypique}, \ref{subsection:action.frobenius.composant.iso} and \ref{subsection:independance.ell}}.
\end{proof}

\begin{remarks}
\begin{subremarks}
    \item Let us recall that the factor corresponding to $[0,1,2,\dots,\allowbreak n-1]$ does not intervene (see \rmrk{\ref{remark:ma.non.zero}} \pg{\pageref{remark:ma.non.zero}}).
    \item The polynomials $\Qa$ depend on $\omega^{\ea}$. See \exmpl{\ref{example:n=4}} \pg{\pageref{example:n=4}}.
    \item When $n$ is a prime number (necessarily odd, as $n \geq 3$), we have $\da = 1$ if $a \neq [0,1,2,\dots,n-1]$, and thus $\omegaeta = 1$; hence, in that case, the numbers $\omegaeta$ don't intervene.
    \item As we mentioned in the introduction, a similar result of factorization was proved by R.~Kloosterman in a slightly different context, see \cite[\crllr{6.10}, \pg{448}]{Kloosterman}. The factorization he obtains is a bit coarser as it involves the polynomials $\Ra(t) = \prod_{\eta}{\Qaomegaeta(t)}$; we refer the reader to \exmpl{\ref{example:n=4}} for an illustration of this phenomenon.
\end{subremarks}
\end{remarks}\bigbreak

We now look how the polynomials $\Qa$ behave over the field $\Da$.

\begin{proposition}\label{result:factorization.over.Da}
The polynomials $\Qa$ factor over $\Da$ as a product of $[\Da:\Q]$ polynomials of degree $\mprimea$.
\end{proposition}

\begin{proof}
As $\Qa(t) = \NormFext{\Da\tensor \Qell[t]}{\Qell[t]}(\Pa(t))$, the polynomial $\Qa$ is the product of the conjugates of $\Pa$.
\end{proof}

The following theorem shows that this factorization is independent of $\ell$.

\begin{theorem}\label{result:Pa.independent.ell}
The polynomials $\Pa$ have coefficients in $\Da$ and are independent of~$\ell$.
\end{theorem}

\begin{proof}
Let us recall that $\Pa(t) = \deth{1-t\va}{\Vaomega/\Da\tensor[\Q]\Qell}$. Using the same argument as in \lmm{\ref{lemme:det.ind.ell=tr.puiss.r.ind.ell}}, we only need to show the independence of $\ell$ of $\traceh{\va^r}{\Vaomega/\Da\tensor[\Q]\Qell}$ for every $r \in \N$.

As $(x,y) \mapsto \TrFext{\Da\tensor[\Q]\Qell}{\Qell}(xy)$ is a non-degenerate bilinear form, the independence of $\ell$ of $\traceh{\va^r}{\Vaomega/\Da\tensor[\Q]\Qell}$ is equivalent to that of the element $\traceh{d\va^r}{\Vaomega/\Qell} \in \Qell$ for all $d \in \Da$; indeed:
\begin{align*}
& \TrFext{\Da\tensor[\Q]\Qell}{\Qell}{(d\traceh{\va^r}{\Vaomega/\Da\tensor[\Q]\Qell})}\\
& \hspace{8em} = \TrFext{\Da\tensor[\Q]\Qell}{\Qell}{(\traceh{d\va^r}{\Vaomega/\Da\tensor[\Q]\Qell})} \\
& \hspace{8em} = \traceh{d\va^r}{\Vaomega/\Qell}.
\end{align*}
Because $d \va^r$ is the map $v \mapsto (\Frobcohom)^r \circ v \circ d$, thanks to \rmrk{\ref{remarque:Het->Hetprim}}, we only need to show the following proposition.
\end{proof}

\begin{proposition}\label{result:tr.va.ind.ell}
Let $X$ be a smooth projective variety over $\Fq$. Let $\groupG$ be a finite subgroup of $\Aut_{\Fq}(X/\Fq)$, $W$ a simple $\Q[\groupG]$-module, $\D$ (the opposite of) its endomorphism ring, and $i$ an integer $\geq 0$. Denote by $V$ the $\D \tensor[\Q] \Qell$-module $\Hom_{\Q[\groupG]}(W,\HetladQell{i}{\Xbar})$ and, given $d \in \D$ and $r \geq 1$, denote by $\alpha$ the endomorphism $v \mapsto (\Frobcohom)^r \circ v \circ d$ of the $\Qell$-vector space $V$. The trace of $\alpha$ is an element of $\Q$ which is independent of $\ell$.
\end{proposition}

\begin{proof}
Denote by $E$ the $\Qell$-vector space $\Hom_{\Q}(W,\HetladQell{i}{\Xbar})$, the action of $\groupG$ on $E$ being $g \cdot v = g^* \circ v \circ g_W^{-1}$ where $g_W$ is the endomorphism of the $\Q$-vector space $W$ induced by $g$. Let $\pi$ be the $\Qell$-linear map from $E$ to itself defined by
\[\pi(v) = \frac{1}{\card{\groupG}} \sum_{g \in \groupG}{g^* \circ v \circ g_W^{-1}}.\]
It is a projection with image $E^G = V$. The map $\beta \colon v \mapsto (\Frobcohom)^r \circ v \circ d$ is an endomorphism of the $\Qell$-vector space $E$ which stabilizes $V$; the endomorphism of  $V$ induced by $\beta$ is $\alpha$ and, because $\pi$ is a projection of $E$ onto $V$, we have
\[\trace(\alpha) = \trace(\pi \circ \beta),\]
where the endomorphism $\pi \circ \beta$ can be written as
\[v \mapsto \sum_{i\in I}{(\Frobcohom)^r \circ g_i^* \circ v \circ f_i,}\]
with $I$ a finite set, $g_i$ some elements of $\groupG$ and $f_i$ some endomorphisms of the $\Q$-vector space $W$, each of them independent of $\ell$. We thus only need to show the following lemma.
\end{proof}

\begin{lemma}
We keep the notations of the previous proposition. If $g \in \groupG$, $f \in \End_{\Q}(W)$ and $r \in \N^*$, then the trace of
\[v \mapsto (\Frobcohom)^r \circ g^* \circ v \circ f\]
considered as an endomorphism of $V$ is an element of $\Q$ independent of $\ell$.
\end{lemma}

\begin{proof}
Let $(e_1,\dots,e_k)$ be a basis of $W$ over $\Q$; the map
\[v \mapsto (v(e_1),\dots,v(e_k))\]
is an isomorphism of the $\Qell$-vector space $V$ onto the $\Qell$-vector space $\HetladQell{i}{\Xbar}^k$. It sends the endomorphism of $V$ given by
\[v \mapsto (\Frobcohom)^r \circ g^* \circ v \circ f\]
to the endomorphism of $\HetladQell{i}{\Xbar}^k$ given by
\[(h_1,\dots,h_k) \mapsto \biggl(\mathinner{}\sum_{i=1}^{k}{a_{i,j}((\Frobcohom)^r \circ g^*)(h_i)}\biggr)_{1 \leq j \leq k},\]
where $(a_{i,j})_{1 \leq i,j \leq k}$ is the matrix of $f$ in the basis $(e_i)_{1 \leq i \leq k}$. Its trace is thus equal to
\[\biggl(\mathinner{}\sum_{i=1}^{k}{a_{i,i}}\biggr) \traceh{(\Frobcohom)^r \circ g^*}{\HetladQell{i}{\Xbar}}.\]
By \cite[\thrm{2.2}, \pg{76}]{Katz.Messing}, it is independent of $\ell$.
\end{proof}

\begin{remark}\label{remarque:Het->Hetprim}
In the previous lemma and proposition, it is possible, when $X$ is a hypersurface, to replace $\HetladQell{n-2}{\Xbar}$ by $\HetladQellprim{n-2}{\Xbar}$ using \lmm{\ref{result:ind.ell:Hiprime}} instead of \cite[\thrm{2.2}, \pg{76}]{Katz.Messing} (indeed, $\Frobcohom$ and each $g^*$, with $g \in G$, extend to $\varproj{n-1}$).
\end{remark}

\subsection{Examples}\label{subsection:exemples}

In this \subsctn{\ref{subsection:exemples}}, we detail the computations for the cases $n=3$, $n=4$, $n=5$, and $n=7$. In all these examples, we use the fact that, when $n$ is prime and $a \neq [0,1,2,\dots,n-1]$, we have $\omega = 1$ and $\da = 1$, hence $\mprimea = \ma$ and $\gammaa/\da = \gammaa$. Let us recall that the degree of $\Qa$ is $(\deg \Pa) [\Da:\Q] = \mprimea\frac{\eulerphi(\na)}{\card{\img\ka}}$. In the tables, the lines appear by decreasing values of $\ma$.

\begin{example}[$n=3$]\label{example:n=3}
This is the simplest non-trivial case. The elements of $\cargroupA$ are, up to permutation, $[0,0,0]$ and $[0,1,2]$. The multiplicity of the latter is zero so only $[0,0,0]$ gives rise to a factor in the zeta function. This factor has degree $\mprimea = 2$ and appears with a power $\gammaa/\da = \gammaa = 1$, so
\[\funzeta{/\Fq}{t} = \frac{\Qaex{[0,0,0]}{1}(t)}{(1-t)(1-qt)}, \qquad \text{with} \quad \deg \Qaex{[0,0,0]}{1}(t) = 2.\]
In fact, in this case, $\dworkhypersurfaceFqbar$ is an elliptic curve, so the previous result doesn't give any new information.
\end{example}

\begin{example}[$n=4$]\label{example:n=4}
Here is a list of the elements of $\cargroupA$ mod the simultaneous actions of $\Sn$ and $\ZnZinv$
\begin{center}\begin{tabular}{c|c|c|c|c}
class of $\classZnZinv{a}$ & $\deg \Qa$ & $\gammaa/\da$ & $\Da$ & $\omega$ \\
\hline
$[0,0,0,0]$ & $3$ &  $1$ & $\Q$ & $1$\\
\hline
$[0,0,2,2]$ & $1$ &  $3$ & $\Q$ & $\pm 1$\\
\hline
$[0,0,1,3]$ & $1$ & $12$ & $\Q$ & $1$ \\
\end{tabular}\end{center}
Consequently, we have the following factorization of the zeta function:
\begin{align*}
\funzeta{/\Fq}{t}
& = \frac{1}{(1-t)(1-qt)(1-q^2t)} \\
& \quad\hspace{2em}\times \frac{1}{\Qaex{[0,0,0,0]}{1}(t) \Qaex{[0,0,2,2]}{1}(t)^3 \Qaex{[0,0,2,2]}{-1}(t)^3 \Qaex{[0,0,1,3]}{1}(t)^{12}}.
\end{align*}
This result is in accordance with the numerical observations of \cite[\sctn{6.1.1}, \pgs{112}{116}]{Kadir.thesis}; let us note that, according to her tables for $q = p = 13$, $17$, $29$, $37$, $41$ (we remind the reader that only the cases $q \equiv 1 \mod 4$ fall in the framework of our study) and $\parameter = 2$, $3$, $2$, $2$, $2$ respectively, we have $\set{\Qaex{[0,0,2,2]}{1}(t),\allowbreak\Qaex{[0,0,2,2]}{-1}(t)} = \set{1-pt,1+pt}$, hence the two polynomials $\Qaex{[0,0,2,2]}{1}$ and $\Qaex{[0,0,2,2]}{-1}$ are not 
generally equal.

This example also illustrate the fact that our method gives a slightly finer factorization than that of \cite{Kloosterman}: instead of finding a factor $\Raex{[0,0,2,2]}^3$ with $\Raex{[0,0,2,2]}$ of degree $2$, we find a factor $\Qaex{[0,0,2,2]}{1}(t)^3 \Qaex{[0,0,2,2]}{-1}(t)^3$ with $\Qaex{[0,0,2,2]}{1}$ and $\Qaex{[0,0,2,2]}{-1}$ of degree $1$; thus, Kloosterman's polynomial $\Raex{[0,0,2,2]}$ factors over $\Q$ as a product of two polynomials of degree $1$.
\end{example}

\begin{example}[Cas $n=5$]\label{example:n=5}
Here are all the elements of $\cargroupA$ (mod the simultaneous actions of $\Sn$ and $\ZnZinv$) which intervene in the zeta function:
\begin{center}\begin{tabular}{c|c|c|c}
class of $\classZnZinv{a}$ & $\deg \Qa[1]$ & $\gammaa/\da$ & $\Da$ \\
\hline
$[0,0,0,0,0]$ & $4$ &  $1$ & $\Q$ \\
\hline
$[0,0,0,1,4]$ & $4$ & $20$ & $\Q(\sqrt{5})$ \\
$[0,0,1,1,3]$ & $4$ & $30$ & $\Q(\sqrt{5})$ \\
\end{tabular}\end{center}
We can thus write:
\[\funzeta{/\Fq}{t} = \frac{\Qaex{[0,0,0,0,0]}{1}(t) \Qaex{[0,0,0,1,4]}{1}(t)^{20} \Qaex{[0,0,1,1,3]}{1}(t)^{30}} {(1-t)(1-qt)(1-q^2t)(1-q^3t)}.\]
Moreover, the polynomials $\Qaex{[0,0,0,1,4]}{1}$ and $\Qaex{[0,0,1,1,2]}{1}$ factor over $\Da = \Q(\sqrt{5})$ into a product of two polynomials of degree $2$ (namely, the corresponding $\Paex{a}{1}$ and its conjugate over $\Q(\sqrt{5})$).

We thus recover (and explain) the numerical observation that Candelas, de la Ossa and Rodriguez-Villegas made in \cite[\tbl{12.1}, \pg{133}]{CdlORV.II}\footnote{As mentioned in the introduction, they only make this observation in the case $\parameter=0$, but their numerical data supports it when $\parameter\neq 0$ and $q \equiv 1 \mod 5$.}.
\end{example}

\begin{example}[Cas $n=7$]\label{example:n=7}
The elements of $\cargroupA$ mod the simultaneous actions of $\Sn$ and $\ZnZinv$ are those given in \exmpl{\ref{sous-exemple:calcul.Da:n=7}} page~\pageref{sous-exemple:calcul.Da:n=7}. We complete the list with the useful informations concerning the factorization of the zeta function.
\begin{center}\begin{tabular}{c|c|c|c}
class of $\classZnZinv{a}$ & $\deg \Qa[1]$ & $\gammaa/\da$ & $\Da$ \\
\hline
$[0,0,0,0,0,0,0]$ &  $6$ &    $1$ & $\Q$ \\
\hline
$[0,0,0,0,0,1,6]$ & $12$ &   $42$ & $\Q(\unityroots{7})^+$ \\
$[0,0,0,0,1,1,5]$ & $24$ &  $105$ & $\Q(\unityroots{7})$ \\
$[0,0,0,1,1,1,4]$ & $12$ &  $140$ & $\Q(\unityroots{7})^+$ \\
$[0,0,0,1,1,6,6]$ & $12$ &  $210$ & $\Q(\unityroots{7})^+$ \\
\hline
$[0,0,0,0,1,2,4]$ &  $6$ &  $210$ & $\Q(\sqrt{-7})$ \\
$[0,0,0,1,1,2,3]$ & $18$ &  $420$ & $\Q(\unityroots{7})$ \\
$[0,0,1,1,3,3,6]$ &  $6$ &  $630$ & $\Q(\sqrt{-7})$ \\
\hline
$[0,0,0,1,2,5,6]$ &  $6$ &  $840$ & $\Q(\unityroots{7})^+$ \\
$[0,0,1,1,3,4,5]$ &  $6$ & $1260$ & $\Q(\unityroots{7})^+$ \\
$[0,0,1,1,2,4,6]$ &  $6$ & $1260$ & $\Q(\unityroots{7})^+$ \\
\end{tabular}\end{center}
As in the preceding cases, from this table, we can easily describe the factorization of the zeta function in the case $n=7$.
\end{example}

\appendix
\section*{Acknowledgments}

I would like to thank my thesis advisor, Joseph Oesterl\'e, for sharing his ideas with me and for the numerous improvements he suggested to the text of the present article. I would also like to thank Luc Illusie for a helpful reference concerning \thrm{\ref{theorem:trace.lefschetz:EP}} as well as Julien Grivaux for his elegant proof of \lmm{\ref{lemme:h*surPn}}.

\section{List of notations}\label{appendix:notations}

\renewcommand*{\pg}[1]{p.~#1}
\begin{center}\begin{longtable}{@{}lp{\textwidth-4.5cm}l}

\multicolumn{3}{@{}l}{\textbf{General notations}} \\*

$\card{E}$ & number of elements of $E$ & \\*


$\Fq$ & finite field with $q$ elements \\

$\Qell$ & field of $\ell$-adic numbers \\

$\overline{\K}$ & algebraic closure of the field $\K$ \\

$\unityroots{n}(\fieldk)$ & set of \nth{$n$} roots of unity belonging to the field $\fieldk$ & \\

$\eulerphi$ & Euler totient function & \\

$\Sn$ & permutation group of $\set{1,\dots,n}$ & \\

$\signature$ & signature (of a permutation) & \\

$\Ind_{H}^{G} \mu$ & representation of $G$ induced by the representation $\mu$ of $H$ \\

$\integerinterval{1}{n}$ & set of integers $k$ satisfying $1 \leq k \leq n$ \\*
\\
%
%
\multicolumn{3}{@{}l}{\textbf{Notations from the \hyperref[section:introduction]{introduction}}} \\*

$\parameter$ & parameter belonging to $\Fqnonzero$ & \pg{\pageref{definition:parameter}} \\*

$\delta_i$ & $\delta_i = 0$ if $i$ is even and $\delta_i = 1$ if $i$ is odd & \pg{\pageref{definition:delta.i}} \\

$\groupA$ & group $\setst{(\zeta_1,\dots,\zeta_n)\in \unityroots{n}(\Fq)^n} {\zeta_1\dots\zeta_n=1}$ quotiented by $\set{(\zeta,\dots,\zeta)}$; is isomorphic to $(\ZnZ)^{n-2}$ & \pg{\pageref{definition:groupA}} \\

$\cargroupA$ & group $\setst{(a_1,\dots,a_n)\in(\ZnZ)^n} {a_1+\dots+a_n=0}$ quotiented by the diagonal $\set{(a,\dots,a)}$; can be identified with the group of characters of $\groupA$ & \pg{\pageref{definition:cargroupA}} \\

$[\zeta_1,\dots,\zeta_n]$ & element of $\groupA$ & \pg{\pageref{definition:elt.groupA}}\\

$[a_1,\dots,a_n]$ & element of $\cargroupA$ & \pg{\pageref{definition:elt.cargroupA}}\\

$\groupG$ & group $\groupA \rtimes \Sn$ & \pg{\pageref{definition:groupG}}\\*
\\
%
%
\multicolumn{3}{@{}l}{\textbf{Notations from \sctn{\ref{section:preliminaries}}}} \\*

$X^f$ & subscheme of fixed point of an automorphism $f$ of $X$ & \pg{\pageref{definition:Xf}} \\*

$\eulerpoincarecaract(X)$ & Euler\namedash Poincar\'e characteristic of a scheme $X$ & \pg{\pageref{definition:euler-poincare.caract}} \\

$\HetladQellinprim{n-2}{X}$ & non-primitive part of the cohomology of a hypersurface of dimension $n-2$; is zero when the dimension is odd & \pg{\pageref{definition:cohom.et.inprim}} \\

$\HetladQellprim{n-2}{X}$ & primitive part of the cohomology of a hypersurface of dimension $n-2$ & \pg{\pageref{definition:cohom.et.prim}} \\*
\\
%
%
\multicolumn{3}{@{}l}{\textbf{Notations from \sctn{\ref{section:Qellbar[A]}}.}} \\*

$k(\zeta)$ & number of $i \in \set{1,\dots,n}$ such that $\zeta_i = \zeta$ & \pg{\pageref{definition:kzeta}}\\

$\ma$ & multiplicity of the character $a$ in the $\Qellbar[\groupA]$-module $\HetladQellbarprim{n-2}{\dworkhypersurfaceFqbar}$ & \pg{\pageref{definition:ma}}\\*
\\
%
%
\multicolumn{3}{@{}l}{\textbf{Notations from \sctn{\ref{section:Qellbar[G]}}.}} \\

$\isotypiccomponentQellbar{a}$ & $a$-isotypic component of the $\Qellbar[\groupA]$-module $\HetladQellbarprim{n-2}{X}$; its dimension is $\ma$ & \pg{\pageref{definition:composantisotypiqueQellbar.a}}\\

$\stabilizerGQellbar{a}$ & stabilizer of $a$ in $\groupG$ & \pg{\pageref{definition:GaQellbar}}\\

$\classSn{a}$ & orbit of $a \in \cargroupA$ under $\Sn$ & \pg{\pageref{definition:classeSna}}\\

$R$ & representative set $\subset \cargroupA$ of the elements of $\Sn \backslash \cargroupA$ & \pg{\pageref{definition:sysclasseSna}}\\
$\SaQellbar$ & stabilizer of $a$ in $\Sn$ & \pg{\pageref{definition:SaQellbar}}\\

$\nprimea$ & generator $\in \integerinterval{1}{n}$ of the set of elements $j \in \ZnZ$ such that $(a_1+j,\allowbreak\dots,\allowbreak a_n+j)$ is a permutation of $(a_1,\dots,a_n)$& \pg{\pageref{definition:nprimea}} \\

$\da$ & integer equal to $n/\nprimea$ & \pg{\pageref{definition:da}} \\

$I(b)$ & set of $i \in \integerinterval{1}{n}$ such that $a_i = b$ & \pg{\pageref{definition:I(b)}}\\

$\genSigmaQellbar$ & element of $\SaQellbar$ belonging to the preimage of a generator of the cyclic group $\SaQellbar/\SaQellbarprime$ & \pg{\pageref{definition:genSigmaQellbar}}\\

$\SaQellbarprime$ & stabilizer in $\Sn$ of a representative $(a_1,\dots,a_n)$ of $a$ in $(\ZnZ)^n$ & \pg{\pageref{definition:SaQellbarprime}}\\

$\gammaa$ & number of permutations of $(a_1,\dots,a_n)$; equal to $[\Sn:\SaQellbarprime]$ & \pg{\pageref{definition:gammaa}}\\

$\SigmaQellbar$ & group generated by $\genSigmaQellbar$; we have $\SaQellbar = \SaQellbarprime \rtimes \SigmaQellbar$ & \pg{\pageref{definition:SigmaQellbar}}\\

$\ja$ & group homomorphism $\SaQellbar \to \nprimea\ZnZ$ defined by ${}^s (a_1,\dots,a_n) = ({a_1+\ja(s)},\allowbreak\dots,a_n+\ja(s))$; satisfies $\ja[ka] = k\ja$ & \pg{\pageref{definition:ja}}\\

$\cargroupAsigma$ & set of elements of $\cargroupA$ fixed by $\sigma \in \Sn$ & \pg{\pageref{definition:cargroupA.fixe.s}}\\

$O_j$ & orbits of a product of $n'$ disjoint cycles of length $d$ & \pg{\pageref{definition:orbites}}\\

$k(\zeta)$ & number of $j \in \set{1,\dots,n'}$ such that $\prod_{i \in O_j}{\zeta_i} = \zeta$; this notation generalizes that from~\pg{\pageref{definition:kzeta}} & \pg{\pageref{definition:kzeta.nprime}}\\

$\mprimea$ & $\mprimea = \ma/\da$ & \pg{\pageref{definition:mprimea}} \\
$\reg$ & regular representation of $\SaQellbar/\SaQellbarprime$ & \pg{\pageref{definition:reg}}\\*
\\
%
%
\multicolumn{3}{@{}l}{\textbf{Notations from \sctn{\ref{section:Qell[G]}}}} \\*
$\fieldcyclgroup{C}$ & cyclotomic field attached to a cyclic group $C$ & \pg{\pageref{definition:corpsgpecycl}}\\*

$\carcyclgroup{C}$ & canonical character of a cyclic group $C$; takes its values in $\fieldcyclgroup{C}$ & \pg{\pageref{definition:cargpecycl}}\\

$\classZnZinv{a}$ & class mod $\ZnZinv$ of $a$ & \pg{\pageref{definition:classeZnZinv}} \\

$\imgclassZnZinv{a}$ & image of the homomorphism $[\zeta_1,\dots,\zeta_n] \mapsto \zeta_1^{a_1}\dots\zeta_n^{a_n}$ & \pg{\pageref{definition:Ea}} \\

$\kerclassZnZinv{a}$ & kernel of the homomorphism $[\zeta_1,\dots,\zeta_n] \mapsto \zeta_1^{a_1}\dots\zeta_n^{a_n}$ & \pg{\pageref{definition:Na}} \\

$\na$& order of $a$ in $\cargroupA$; equal to the order of the group generated by $a_i-a_{i'}$; also equal to the number of elements of the image of the character $a$ & \pg{\pageref{definition:na}} \\

$\fieldcyclgroupZnZinv{a}$ & cyclotomic field attached to the cyclic group $\groupA/\kerclassZnZinv{a}$; its dimension over $\Q$ is $\eulerphi(\na)$; only depends on $\classZnZinv{a}$ & \pg{\pageref{definition:corpsKa}}\\

$\carcyclgroupZnZinv{a}$ & canonical character of the cyclic group $\groupA/\kerclassZnZinv{a}$ considered as a character of $\groupA$; takes values in $\fieldcyclgroupZnZinv{a}$ and satisfies $\carcyclgroupZnZinv{ka} = \carcyclgroupZnZinv{a}^k$ & \pg{\pageref{definition:cara}}\\

$\fa$ & generator of the group generated by $a_i-a_{i'}$; satisfies $\nprimea = \ea \fa$, $n = \ea \fa \da$ and $n = \na \fa$ & \pg{\pageref{definition:fa}} \\

$\SaQell$ & fixator of $\classZnZinv{a}$ in $\symmetricgroup_n$ & \pg{\pageref{definition:SaQell}} \\

$\ka$ & group homomorphism $\SaQell \to \ZnaZinv$ defined by $\sigmaa = \ka(\sigma)a$; only depends on $\classZnZinv{a}$ & \pg{\pageref{definition:ka}} \\

$\ea$ & integer such that $\nprimea = \ea \fa$; satisfies $\na = \ea \da$ and $n = \ea \fa \da$ & \pg{\pageref{definition:ea}} \\

$(\usigma,\vsigma)$ & if $\sigma \in \SaQell$, unique pair $(\usigma,\vsigma) \in \ZnaZ \times \ZnaZinv$ such that $a_{\sigma(i)} = \vsigma a_i + \usigma \fa$ & \pg{\pageref{definition:u.v}} \\

$\phi$ & group homomorphism $\SaQell \to \ZnaZ \rtimes \ZnaZinv$, $\sigma \mapsto (\usigma,\vsigma)$; we have $\vsigma = \ka(\sigma)$ and $\fa\usigma = \ja(\sigma)$ & \pg{\pageref{definition:phi.sigma.u.v}} \\

$\theta_{v}$ & automorphism of the field $\Ka$ sending the \nth{$\na$} roots of unity to their \nth{$v$} power & \pg{\pageref{definition:morphisme.theta.v}} \\

$\omega$ & \nth{$\na$} root of unity & \pg{\pageref{definition:omega}} \\

$\mua$ & representation $(\zeta,\sigma) \mapsto \cara(\zeta)\signature(\sigma)\omega^{\usigma}\theta_{\vsigma}$ of $\groupA \rtimes \SaQell$ in $\Ka$ & \pg{\pageref{definition:mu.a.omega}} \\

$\Mua$ & $\Q[\groupA \rtimes \SaQell]$-module $\Ka$ given by $\mua$; up to isomorphism, only depends on $\omega^{\ea}$, not on $\omega$ & \pg{\pageref{definition:Mua}} \\

$\Waomega$ & $\Q[\groupG]$-module simple $\Ind_{\groupA \rtimes \SaQell}^{\groupG}{\Mua}$ & \pg{\pageref{definition:Wa}} \\

$\Da$ & (opposite of the) endomorphism ring of $\Waomega$~; we have $\Da \subset \Ka$ (hence $\Da$ is commutative) and $\dim_{\Q} \Da = \frac{\eulerphi(\na)}{\card{\img\ka}}$ & \pg{\pageref{definition:Da}} \\*
\\
%
%
\multicolumn{3}{@{}l}{\textbf{Notations from \sctn{\ref{section:fact.zeta}}}} \\*

$\Vaomega$ & $\Hom_{\Q[\groupG]}(\Waomega,\HetladQellprim{n-2}{\dworkhypersurfaceFqbar})$; is a free $\Da \tensor[\Q] \Qell$-module of rank $\mprimea$; $\Waomega \tensor[\Da] \Vaomega$ identifies with the $\Waomega$-isotypic component $\isotypiccomponentQell{a}$ of the $\Q[\groupG]$-module $\HetladQellprim{n-2}{\dworkhypersurfaceFqbar}$ & \pg{\pageref{definition:Va}}\\*

$\isotypiccomponentQell{a}$ & $\Waomega$-isotypic component of the $\Q[G]$-module $\HetladQellprim{n-2}{\dworkhypersurfaceFqbar}$; is isomorphic to $\Waomega \tensor[\Da] \Vaomega$ & \pg{\pageref{definition:Ha}}\\

$\va$ & endomorphism of the $\Da \tensor[\Q] \Qell$-module $\Vaomega$ such that $\Frobcohom|\Waomega \tensor[\Da] \Vaomega = \Id \tensor \va$ & \pg{\pageref{definition:va}} \\

$\Pa$ & polynomial $\deth{1-t\va}{\Vaomega/\Da\tensor[\Q]\Qell}$ having degree $\mprimea$; has coefficients in $\Da$ and is independent of~$\ell$ & \pg{\pageref{definition:Pa}} \\

$\Qa$ & polynomial $\NormFext{\Da\tensor \Qell[t]}{\Qell[t]}(\Pa(t))$ having degree $\mprimea\frac{\eulerphi(\na)}{\card{\img\ka}}$ and coefficients in $\Q$; is independent of $\ell$ & \pg{\pageref{definition:Qa}}
\end{longtable}\end{center}

\section{Formulas}\label{appendix:formulas}

Here is a list of the most important formulas established throughout this article.

$\phantombigstrut n = \nprimea \da = \ea \fa \da = \na\fa$, $\nprimea = \ea \fa$, and $\na = \ea \da$.

$\phantombigstrut[\Sn : \SaQellbarprime] = \gammaa$ (number of permutations of $(a_1,\dots,a_n)$)

$\phantombigstrut[\Sn : \SaQellbar] = \frac{\gammaa}{\da}$

$\phantombigstrut[\Sn : \SaQell] = \frac{\gammaa}{\card{(\img\ka)}\da}$

$\phantombigstrut[\SaQellbar : \SaQellbarprime] = \da$

$\phantombigstrut[\SaQell : \SaQellbar] = \card{\img\ka}$\qquad(in fact, $\SaQell/\SaQellbar = \img\ka$)

$\phantombigstrut[\SaQell : \SaQellbarprime] = \da \card{\img\ka}$

$\phantombigstrut\dim \isotypiccomponentQellbar{a} = \ma$

$\phantombigstrut\dim \mua = \dim \Mua = \dim \Ka = \eulerphi(\na)$

$\phantombigstrut\dim \Mua^{\ma '} = \ma ' \eulerphi(\na)$

$\phantombigstrut\dim \Waomega^{\mprimea} = \dim \Ind_{\groupA \rtimes \SaQell}^{\groupG}{\Mua^{\ma '}} = \ma ' \eulerphi(\na)[\Sn : \SaQell] = \ma '\frac{\eulerphi(\na)}{\card{\img\ka}}\frac{\gammaa}{\da}$

$\phantombigstrut\dim \bigoplus_{\eta \in \unityroots{\da}(\Ka)}{\Ind_{\groupA \rtimes \SaQell}^{\groupG}{\Muaomegaeta^{\ma '}}} = \ma\frac{\eulerphi(\na)}{\card{\img\ka}}\frac{\gammaa}{\da}$.

$\phantombigstrut\dim_{\Q}{\Da} = \frac{\eulerphi(\na)}{\card{\img\ka}}$

$\phantombigstrut\dim_\Q(\Waomega) = \frac{\eulerphi(\na)}{\card{\img\ka}} \frac{\gammaa}{\da} = [\Sn:\SaQellbar][\Da:\Q]$.

$\phantombigstrut\dim_{\Da}(\Vaomega) = \mprimea$.

$\phantombigstrut\dim_{\Q}(\isotypiccomponentQell{a}) = \mprimea \frac{\eulerphi(\na)}{\card{\img\ka}} \frac{\gammaa}{\da} = \mprimea{} [\Sn:\SaQellbar][\Da:\Q]$

$\phantombigstrut\begin{aligned} \dim_{\Qell}(\HetladQellprim{n-2}{\dworkhypersurfaceFqbar})
& = \sum_{a \in \ZnZinv \times \Sn \backslash \cargroupA}{\mprimea \frac{\eulerphi(\na)}{\card{\img\ka}}\gammaa} = \sum_{a \in \ZnZinv \times \Sn \backslash \cargroupA}{\ma \frac{\eulerphi(\na)}{\card{\img\ka}}\frac{\gammaa}{\da}}.\end{aligned}$

$\phantombigstrut\deg \Pa = \mprimea$

$\phantombigstrut\deg \Qa = (\deg \Pa) [\Da:\Q] = \mprimea\frac{\eulerphi(\na)}{\card{\img\ka}}$

\end{document}